%% file: renor_2.tex
\documentclass[a4paper,11pt]{article}

\usepackage{amsmath}
\usepackage{amsfonts}
\usepackage{amssymb}
\usepackage{amsthm}
\usepackage{graphicx}
\usepackage{verbatim}
\usepackage{color}
\usepackage{authblk}
\usepackage{setspace}

\newtheorem{thm}{Theorem}[section]
\newtheorem{cor}[thm]{Corollary}
\newtheorem{lem}[thm]{Lemma}
\newtheorem{prop}[thm]{Proposition}
\theoremstyle{definition}
\newtheorem{defin}[thm]{Definition}

\newtheorem{rem}[thm]{Remark}

\newtheorem{conj}{Conjecture}[part]

\newcommand{\deltabf}{\text{\boldmath $\delta$}}
\newcommand{\eps}{\varepsilon}
\newcommand{\B}{\mathbb{B}}
\newcommand{\C}{\mathbb{C}}

\newcommand{\N}{\mathbb{N}}

\newcommand{\R}{\mathbb{R}}
\newcommand{\T}{\mathbb{T}}
\newcommand{\W}{\mathbb{W}}
\newcommand{\Z}{\mathbb{Z}}

\newcommand{\BB}{\mathcal{B}}
\newcommand{\DD}{\mathcal{D}}

\newcommand{\LL}{\mathcal{L}}
\newcommand{\MM}{\mathcal{M}}

\newcommand{\RR}{\mathcal{R}}
\newcommand{\RHH}{\mathcal{RH}}
\newcommand{\TT}{\mathcal{T}}
\newcommand{\XX}{\mathcal{X}}

\newcommand{\Ev}{\operatorname{Ev}}
\newcommand{\tv}{\tilde{v}}

\author[1]{Pau Rabassa}
\author[2]{Angel Jorba}
\author[2]{Joan Carles Tatjer}
\affil[1] {\mbox{Johann Bernoulli Institute for Mathematics and Computer Science,}
\centerline{ \mbox{University of Groningen, Groningen, The Netherlands}}
\newline
\mbox{E-mail: {\tt paurabassa@gmail.com}}
\vspace{2mm}}
\affil[2]{\mbox{Departament of Matem\`atica Aplicada i An\`alisi,}
\mbox{Universitat de Barcelona, Barcelona, Spain}
\mbox{E-mails: {\tt angel@maia.ub.edu}, {\tt jcarles@maia.ub.es}}}

\date{}

\title{
Towards a renormalization theory for quasi-periodically forced 
         one dimensional maps {II}. {A}symptotic behavior of 
         reducibility loss bifurcations\thanks{This work has been supported 
         by the MEC grant MTM2009-09723 and the CIRIT grant 2009 SGR 67. 
         P.R. has been partially supported by the PREDEX project, funded
         by the Complexity-NET: {\tt www.complexitynet.eu}.
}}

\begin{document}
\maketitle
\begin{abstract}
In this paper we are concerned with quasi-periodic
forced one dimensional maps.
We consider a two parametric family
of quasi-periodically forced maps such that the one dimensional 
map (before forcing) 
is unimodal and it has a full cascade of period doubling bifurcations. 
Between one period doubling and the next one it is known that there 
exist a parameter value where the $2^n$-periodic orbit is superatracting.
In a previous work we proposed an 
extension of the one-dimensional (doubling) renormalization 
operator to the quasi-periodic case. We proved that, if the family 
satisfies suitable hypotheses, the two parameter family has two 
curves of reducibility loss bifurcation around these parameter 
values. In the present work we study the asymptotic behavior 
of these bifurcations when $n$ grows to infinity. We show 
that the asymptotic behavior depends 
on the Fourier expansion of the quasi-periodic coupling of the family. 
The theory developed here provides a theoretical explanation 
to the behavior that can be observed numerically. 
\end{abstract}

\tableofcontents


\section{Introduction}
\label{section intro} 

This is the second of a series of papers (together with \cite{JRT11a,JRT11c})
where we propose an extension of the one dimensional renormalization
theory for the case of quasi-periodic forced maps.  Each of these
papers is self contained, but highly interrelated with the others. 
An more detailed exposition can be found in \cite{Rab10}.
In \cite{JRT11a} we give the definition of the operator for 
the case of quasi-periodic maps and we use it to prove the 
existence of reducibility loss bifurcations when the coupling parameter goes to
zero. In this paper we use the results obtained there 
to study the asymptotic behavior of these bifurcations 
when the period of the attracting set goes to infinity. 
Our quasi-periodic extension of the renormalization
operator is not complete in the sense that
several conjectures must be assumed. In \cite{JRT11c} we include
the numerical evidence which support our conjectures and we show
that the theoretical results agree with the behavior observed numerically. 
In \cite{JRT11c} we also include a numerical study of the asymptotic 
behavior of the reducibility loss bifurcations which will be summarized 
in the forthcoming section \ref{sec:num obs}. 

The classic one dimensional renormalization theory provides an explanation 
to the behavior observed in the cascades of period doubling bifurcations.
Concretely, given a typical one parametric family for unimodal maps
$\{f_\alpha\}_{\alpha\in I}$  one  observes
numerically that there exists a sequence of parameter
values $\{d_n\}_{n\in \N} \subset I$ such that,
the attracting periodic orbit of the map undergoes a period doubling
bifurcation. Between one period doubling and the next one there
exists also a parameter value $s_n$, for which the critical point
of $f_{s_n}$ is a periodic orbit with period $2^n$. One can
also observe 
that
\begin{equation}
\label{universal limit sumicon}
\lim_{n\rightarrow \infty} \frac{d_n - d_{n-1}} {d_{n+1} - d_{n}} = 
\lim_{n\rightarrow \infty} \frac{s_n - s_{n-1}} {s_{n+1} - s_{n}} = \deltabf = \texttt{ 4.66920...}. 
\end{equation}

This reveals two important phenomena.  The first one is the self-renormalizable 
structure of the bifurcation diagram. Since the limit converges, it indicates that 
there exists a scale factor of $\deltabf$ between one bifurcation and the next.
The second one is the universality, in the sense that the limit $\deltabf$ 
does not depend on the family considered. 

Renormalization theory provides a theoretical explanation to 
this phenomenon. The literature on this topic is quite extensive, 
some remarkable works are  \cite{Fei78,Fei79,Lan82,Sul92,Lyu99,FEM06}, we also  
refer the reader to the books \cite{MvS93,McM94} and references therein. 

In this paper we are interested in the analog of renormalization 
and universality problem for the case of quasi-periodic forced 
one dimensional maps. 
In \cite{JRT11p2} we have given numerical evidences of 
self-similarity of the bifurcation diagram and universality. 
These numerical evidences are described in section \ref{sec:num obs} below.  
In this paper we provide a theoretical explanation to the 
behavior observed numerically.

\subsection{Numerical observations on renormalization and universality for 
quasi-periodically forced maps} 
\label{sec:num obs}

\begin{figure}[t]
\centering
\resizebox{16cm}{!}{
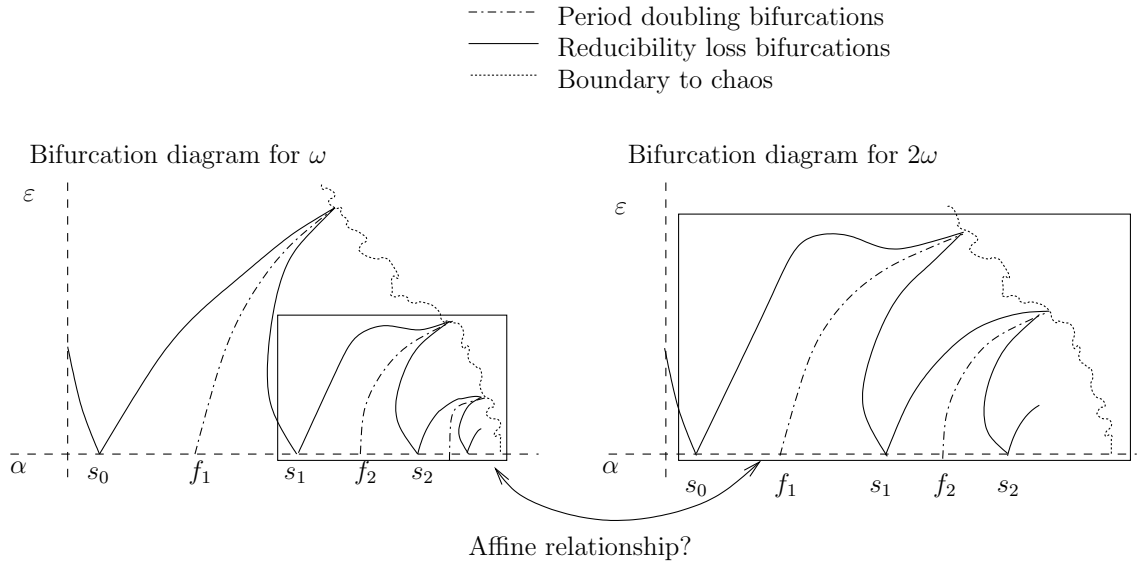
\phantom{aaaaa}
}
\caption{ 
Schematic representation of the bifurcations diagram
of the Forced Logistic Map, for rotation number equal to 
$\omega$ (left) and $2\omega$ (right). See the text for more 
details.}
\label{Esquema BifPar}
\end{figure}

Consider $\{g_{\alpha,\eps}\}_{(\alpha,\eps)
\in J\subset\R^2}$ a two parametric family of 
quasi-periodic maps in the cylinder $\T\times \R$, 
such that it has the form 
\begin{equation}
\label{q.p- family}
\left.
\begin{array}{rcl}
\bar{\theta} & = & \theta + \omega  ,\\
\bar{x} & = & f_\alpha(x) +  \eps h_{\alpha,\eps}(\theta,x)  ,
\end{array}
\right\}
\end{equation}
with $\omega$ a Diophantine number, $\alpha$ and $\eps$ parameters, 
$h$ a periodic with respect $\theta$ and
$\{f_\alpha\}_{\alpha\in J}$  a family of one dimensional maps having 
a complete cascade of period doubling bifurcations as the family 
described before. As before, let 
 $\{d_n\}_{n\in \N} \subset I$ be the parameter values where 
the attracting periodic orbit of the map undergoes a period doubling
bifurcation and $\{s_n\}_{n\in \N} \subset I$ the values for which the critical point
of $f_{s_n}$ is a periodic orbit with period $2^n$. 
The paradigmatic example for this type of maps is the 
Forced Logistic Map (FLM for short), where the uncoupled one dimensional 
family is the logistic map, $f_\alpha (x) = \alpha x(1-x)$
with $\alpha \in [0,4]$. Nevertheless 
the results that we obtain are applicable to a wider class
of maps. 

In \cite{JRT11p} we computed some bifurcation diagrams in terms
of the dynamics of the attracting set. We have  taken
into account different properties of the attracting set, as the 
Lyapunov exponent  and, in the case of having
a periodic invariant curve, its period and reducibility.
The reducibility loss of an invariant curve is not a bifurcation 
in the classical sense, it is only a change in the spectral properties
of the transfer operator associated to the curve (see \cite{JT08}). 
Despite of this, it can be characterized as a bifurcation
(see definition 2.3 in \cite{JRT11p}) and it will be considered 
as such for the rest of this paper. The numerical
computations in the cited work reveal that the parameter
values for which the invariant curve doubles its period are
contained in regions of the parameter space where the invariant
curve is reducible, as sketched in figure \ref{Esquema BifPar}.
Taking into account the properties of universality and self 
renormalization of the Logistic Map, one might look for similar
phenomena in the bifurcation diagram of the FLM.

Let $s_n$ be the parameter value where the critical point of the 
uncoupled family $\{f_\alpha\}_{\alpha\in I}$ is periodic 
with period $2^n$. Numerical computations (see \cite{JRT11p}) 
revealed that from every parameter  value $(\alpha,\eps) = (s_n, 0)$
two curves are  born. These curves correspond to
reducibility-loss bifurcations of the $2^n$-periodic invariant curve.
In \cite{JRT11a} we proved that these curves really exist under
suitable hypotheses. Assume that these two curves can be locally expressed as
$(s_n + \alpha_n'\eps + O(\eps^2),\eps)$  and 
$(s_n + \beta_n' \eps + O(\eps^2),\eps)$. Numerical experiments 
in \cite{JRT11p,JRT11c}  show that the slopes depend on $\omega$, 
i. e. $\alpha_n'=\alpha_n'(\omega)$ and  $\beta_n'=\beta_n'(\omega)$), 
and also show that $\beta_n'(\omega) = - \alpha_n'(\omega)$ for 
the examples studied numerically. In \cite{JRT11a} we
 give explicit expressions of this slopes in terms 
of the quasi-periodic forced renormalization operator, 
for both $\alpha_n'(\omega)$ and $\beta_n'(\omega)$. In this paper we focus 
only on $\alpha_n'(\omega)$, but the discussion for $\beta_n'(\omega)$
is completely analogous. 

The slopes $\alpha_n'(\omega)$ can be used for the numerical detection of universality and
self-renormalization phenomena. If the bifurcation diagram is 
self renormalizable one should have that $\alpha_n'(\omega)/ 
\alpha_{n-1}'(\omega)$ converges to a constant. In general, 
this is not true  due to the fact that when the 
period is doubled, the rotation number of the system also is. 
Then one should look  for renormalization properties between 
the bifurcation diagram of the family for rotation number $\omega$ 
and the bifurcation diagram of the same family for rotation number
 $2\omega$. This is sketched in figure \ref{Esquema BifPar}. 
In \cite{JRT11p2} we do a numerical study for the case of the 
Forced Logistic Map and some modifications of it. Concretely we consider 
the family of maps in the cylinder $\T\times \R$ defined by: 
\begin{equation}
\label{FLM family}
\left.
\begin{array}{rcl}
\bar{\theta} & = & \theta + \omega  ,\\
\bar{x} & = &  \alpha x(1-x) + \eps g(\theta,x)  ,
\end{array}
\right\}
\end{equation}
with $\omega$ a Diophantine number. 

In \cite{JRT11p2} we  did  the following discoveries. 

\begin{itemize}

\item {\bf First numerical observation:}  the sequence $\alpha_n'(\omega)/ 
\alpha_{n-1}'(\omega)$ is not convergent in $n$. But, for $\omega$ fix, 
one obtains the same sequence for any family of quasi-periodic forced maps, with 
a quasi-periodic forcing of the type $g(\theta,x) =  f_1(x) \cos(\theta) 
+ f_2(x) \sin(\theta)$. 
\item {\bf Second numerical observation:}  the sequence $\alpha_n'(\omega)/ 
\alpha_{n-1}'(2\omega)$ is convergent in $n$ when 
the quasi-periodic forcing of the type $g(\theta,x) = f_1(x) \cos(\theta) 
+ f_2(x) \sin(\theta)$. The limit depends on $\omega$ 
and on the particular family considered. 
\item {\bf Third numerical observation:} the two previous observations 
are not true when the quasi-periodic forcing is of the 
type $g_\eta(\theta,x) = f_1(x) \cos(\theta) + \eta f_2(x) \cos(2\theta)$
when $\eta\neq 0$. But the sequence $\alpha_n'(\omega)/ 
\alpha_{n-1}'(2\omega)$ associated to the map (\ref{FLM family}) 
with $g= g_\eta$ is $\eta$-close to the same maps with $g= g_0 $
\end{itemize}

In this paper we give a theoretical explanation in terms of 
the dynamics of the quasi-periodic  renormalization operator.
In section \ref{section review q-p renor} we review the concepts 
and results from \cite{JRT11a} that are necessary for this. 
In section \ref{subsection reduction to the dynamics} we reduce the 
study of the asymptotic behavior of the sequences $\alpha_n'(\omega)/ 
\alpha_{n-1}'(\omega)$ to the dynamics of the quasi-periodically 
forced renormalization operator. In sections \ref{subsection 
explanation 1st observation}, \ref{subsection 
explanation 2nd observation} and \ref{subsection 
explanation 3rd observation}
 we give a theoretical explanation to each of the three numerical 
observations described above.

\section{Review on quasi-periodic renormalization} 
\label{section review q-p renor}


Here we summarize the ideas and results developed in \cite{JRT11a} which are 
essential for the discussion. 
Consider a 
quasi-periodic forced map like 
\begin{equation}
\label{q.p. forced system interval}
\begin{array}{rccc}
F:& \T\times I &\rightarrow & \T \times I \\
  & \left( \begin{array}{c} \theta \\ x \end{array}\right)
  & \mapsto
  & \left( \begin{array}{c} \theta + \omega \\ f(\theta,x) \end{array}\right),
\end{array} 
\end{equation}
with $f\in C^r(\T\times I ,I)$. To define the 
renormalization 
operator it is only necessary that $r \geq 1$. 
For simplicity the exposition done here is restricted to the analytic case.  
Along section \ref{section definition and basic properties} 
it is not necessary to require $\omega$ to be  Diophantine, but 
it will be necessary in section \ref{chapter application}. 

Note that the map $F$ (\ref{q.p. forced system interval}) is 
completely determined by the couple $(\omega, f)$. 
From now on we consider $\omega$ fixed and we focus only on the 
function $f$. The definition of the operator is done in a perturbative way, 
in the sense that it is only applicable to maps 
$f(\theta,x) = g(x) + h(\theta,x)$ with $g$ renormalizable in 
the one dimensional case and $h$ small. 

\subsection{Definition of the operator and basic properties} 
\label{section definition and basic properties}

\subsubsection{Preliminary notation} 

Let $\W$ be an open set in the complex plane containing the
interval $I_\delta=[-1-\delta, 1+\delta]$ and let
$\B_\rho = \{z = x + i y\in \C \text{ such that } |y| < \rho\}$.   Then 
consider  $\BB=\BB(\B_\rho,\W)$ the space of functions
$f: \B_\rho \times  \W \rightarrow \C$ such that:
\begin{enumerate}
\item  $f$ is holomorphic in $\B_{\rho}\times \W$ and continuous
in the closure of  $\B_{\rho}\times \W$.
\item $f$  is real analytic. 

\item  $f$ is $1$-periodic in the first variable, i. e.
$f(\theta +1,z) = f(\theta,z)$ for any $(\theta,z) \in \B_{\rho}
\times\W$.
\end{enumerate}
This space, endowed with the supremum norm, is a Banach space. 

Let $\RHH(\W)$  denote the space of functions 
real analytic functions such that are holomorphic
in $\W$, continuous in the closure of  $\W$. 
This is also a Banach space with the supremum norm. 

Consider the operator
\begin{equation}
\label{equation projection 0}
\begin{array}{rccc}
p_0:& \BB &\rightarrow & \RHH(\W) \\
\displaystyle \rule{0pt}{3ex} & f(\theta,x) & \mapsto &  \rule{0pt}{3ex}
\displaystyle  \int_{0}^{1} f(\theta, x) d\theta .
\end{array} 
\end{equation}
Let $\BB_0$ the natural inclusion of $\RHH(\W) $ 
into $\BB$ then we have that $p_0$  as a map from $\BB$ to $\BB_0$ 
is a projection  ($(p_0)^2 = p_0$).

\subsubsection{Set up of the one dimensional renormalization operator. }

First let us give a concrete definition of the 
one dimensional renormalization operator before extending 
it to the quasi-periodic case. Actually, we tune the 
definition of the operator given in \cite{Lan82} in order to be able 
to add a quasi-periodic perturbation.

Given a small value $\delta$, let  $\MM_\delta$ denote the 
subspace of $\RHH(\W)$ formed by the even functions $\psi$ which 
send the interval $I_\delta=[-1-\delta,1+\delta]$ into itself, and such that
$\psi(0)=1$ and  $x \psi'(x) <0$ for $x\neq 0$.

Set $a=\psi(1)$, $a'= (1+\delta)a$ and $b'=\psi(a')$. We 
can define $\DD(\RR_\delta)$ as the set of $\psi \in \MM_\delta$ 
such that $a<0$, 
$1> b'>-a'$,
and  $\psi(b') <- a'$.

We define the renormalization operator, $\RR_\delta : 
\DD(\RR_\delta) \rightarrow \MM_\delta$ as 
\begin{equation}
\label{renormalization operator lanford}
\RR_\delta(\psi) (x) =  \frac{1}{a} \psi \circ \psi (a x).
\end{equation}
where $a=\psi(1)$.

For maps $\psi \in \DD(\RR_\delta)$ such that $\psi\left(a\W\right)\subset \W$
we have that $\RR_\delta(\psi)$ is well defined.

For convenience, we introduce the following working hypothesis. 

\begin{description}
\item[H0)] There exists an open set $\W\subset \C$ containing 
$I_\delta$ and a function $\Phi\in \BB \cap \XX_0$ such 
that $\phi= p_0(\Phi) $ is a fixed point of the  
renormalization operator $\RR_\delta$  and such that 
the closure of both $a\W$ and $\phi(\Phi)(a \W)$ is 
contained in $\W$ (with $a:=\Phi(1)$). 
\end{description}

In \cite{Lan92}, it is claimed that the hypothesis {\bf H0} is satisfied 
by the set
\[
\left\{ z\in \C \text { such that } |z^2 -1| < \frac{5}{2} \right\}.
\]

This set used by Lanford is more convenient in his study
since he works in the
set of even holomorphic functions. In the numerical computations
from \cite{JRT11c} we use as $\W$ the disc centered at
$\frac{1}{5}$ with radius $\frac{3}{2}$,
and we check the hypothesis {\bf H0} numerically
(without rigorous bounds).

\subsubsection{Definition of the renormalization operator for  
quasi-periodically forced maps}

Consider the space $\XX\subset \BB$ defined as: 
\[
\XX= \{ f \in C^r(\T \times I_\delta, I_\delta) | \thinspace 
 p_0(f) \in \MM_\delta\}. 
\]
Consider also the decomposition $\XX=\XX_0 \oplus \XX_0^c$ given by the projection 
$p_0$. In other words, we have $\XX_0=\{f \in \XX \thinspace |\text{ } p_0(f)=f\}$ and 
$\XX^c_0=\{f \in \XX \thinspace | \text{ } p_0(f)=0\}$. Note 
that from the definition of $\XX$ follows that $\XX_0$ is an isomorphic 
copy of $\MM_\delta$.  

Given a function $g\in \XX$, we  
define 
the {\bf quasi-periodic renormalization} of $g$ as 
\begin{equation}
\label{operator tau}
[\TT_\omega(g)](\theta,x) :=  \frac{1}{\hat{a}} g(\theta + \omega,
g(\theta, \hat{a}x)),
\end{equation}
where $\displaystyle \hat{a} = \int_{0}^{1} 
g(\theta, 1) d\theta$. 

Then we have that there exist a set $\DD(\TT)$, 
open in $\left(p_0 \circ \TT_\omega\right)^{-1} (\MM_\delta)$, 
where the operator is well defined, in the sense that 
$\hat{a} \neq 0$.  
Moreover this set contains  $\DD_0(\TT)$, the inclusion of $\DD(\RR)$ in 
$\BB$. By definition we have that $\TT_\omega$ restricted 
to $\DD_0(\TT)$ is isomorphically conjugate to $\RR$, therefore 
the fixed points of $\RR$ extend to fixed points of $\TT_\omega$. 
Assume that {\bf H0} holds and let $\Phi$ be 
the fixed point given by this hypothesis. Then we have that there exists 
$U\subset \DD(\TT)\cap\BB$, an open neighborhood of $\Phi$, 
such that $\TT_\omega :U\rightarrow \BB$ is well defined. 
Moreover we have that $\TT_\omega$ is Fr\'echet differentiable for any  $\Psi \in U$. 

\subsubsection{Fourier expansion of $D\TT_\omega(\Psi)$.} 
\label{section The Fourier expansion of DT} 

Let $\Psi$ be a function in a neighborhood of $\Phi$ (given in hypothesis
{\bf H0}) where $\TT_\omega$ is differentiable. Additionally 
assume that $\Psi\in \DD_0(\TT_\omega)$.

Given a function $f\in \BB$ we can consider its complex Fourier
expansion in the periodic variable
\begin{equation}
\label{fourier expansion complex}
f(\theta,z)= \sum_{k\in \Z} c_k(z) e^{2\pi k\theta i }, 
\end{equation}
with
\[
c_k(z)= \int_{0}^{1} f(\theta,z) e^{-2\pi k\theta i} d\theta.
\]

Then we have that $D\TT_\omega$ ``diagonalizes''
with respect to the complex Fourier expansion, in the
sense that we have 
\begin{equation}
\label{equation Fourier expansion differential renormalization operator}
\left[D\TT_\omega(\Psi) f\right](\theta,z) = D\RR_\delta[c_0](z) 
+  \sum_{k\in \Z\setminus\{0\}} \left([L_1(c_k)](z) + 
[L_2(c_k)](z) e^{2\pi k\omega i}\right) e^{2\pi k\theta i },
\end{equation}
where
\[
\begin{array}{rccc}
L_1: & \RHH(\W) & \rightarrow & \RHH(\W)  \\
  &   g(z)  & \mapsto & 
\displaystyle  \frac{1}{a} \psi'\circ\psi(a z) g(az),
\end{array} 
\]
and 
\[
\begin{array}{rccc}
L_2: & \RHH(\W) & \rightarrow & \RHH(\W)  \\   &   g(z)  & \mapsto &
\displaystyle  \frac{1}{a} g\circ\psi(a z),
\end{array}
\] with $\psi=p_0(\Psi)$ and 
$a=\psi(1)$. 

An immediate consequence of this diagonalization is the following. Consider 
\begin{equation}
\label{equation spaces bbk}
\BB_k:= \big\{ f \in B |\text{ } f(\theta, x) = u(x) \cos(2\pi k \theta) + 
v(x) \sin(2\pi k\theta), \text{ for 
some } u,v\in \RHH(\W)\big\},
\end{equation} 
then we have that the spaces $\BB_k$ are
invariant by $D\TT(\Psi)$ for any $k>0$. 

Moreover $D\TT_\omega (\Psi)$ restricted to $\BB_k$ is
conjugate to $\LL_{k\omega}$, where $\LL_\omega$ is the defined 
as 
\begin{equation}
\label{equation maps L_omega}
\begin{array}{rccc}
\LL_\omega: &  \RHH(\W)\oplus  \RHH(\W) & \rightarrow 
&  \RHH(\W)\oplus  \RHH(\W)  \\ \\
& \left( \begin{array}{c} u \\v \end{array} \right) &
\mapsto &  \left( \begin{array}{c}  L_1(u) \\ L_1(v)
\end{array} \right) +
\left( \begin{array}{cc}  \cos(2\pi \omega) & - \sin(2\pi \omega)  \\
\sin(2\pi \omega) & \cos( 2\pi \omega) 
\end{array} \right)  \left( \begin{array}{c}  L_2(u) \\
L_2(v) 
\end{array} \right) . 
\end{array}
\end{equation}

Then we have that the understanding of the 
derivative of the renormalization operator in $\BB$ is equivalent 
to the study of the operator $\LL_\omega$ for any $\omega\in \T$.

\subsubsection{Properties of $\LL_\omega$} 

Given a value $\gamma \in \T$, consider the rotation $R_{\gamma}$ defined
as
\begin{equation}
\label{equation rotation rgamma}
\begin{array}{rccc}
R_\gamma: &  \RHH(\W)\oplus  \RHH(\W) & \rightarrow
&  \RHH(\W)\oplus  \RHH(\W)  \\ \\
& \left( \begin{array}{c} u \\v \end{array} \right) &
\mapsto &
\left( \begin{array}{cc}  \cos( 2\pi \gamma) & - \sin(2\pi \gamma)  \\
\sin(2\pi \gamma) & \cos(2\pi \gamma)
\end{array} \right)  \left( \begin{array}{c} u \\ v
\end{array} \right) ,
\end{array}
\end{equation}
then we have that $\LL_\omega$ and $R_\gamma$ commute
for any $\omega, \gamma \in \T$. 

This has some consequences on the spectrum of $\LL_\omega$. Concretely 
we have that 
any eigenvalues of $\LL_\omega$ (different from zero) is either real with 
geometric multiplicity even, or a pair of complex conjugate eigenvalues. 
On the other hand $\LL_\omega$ depends analytically
on $\omega$, which (using theorems 
III-6.17 and VII-1.7 of \cite{Kat66}),  imply that 
(as long as the eigenvalues of $\LL_\omega$ are different) 
the eigenvalues and their associated eigenspaces depend 
analytically on the parameter $\omega$. 

Finally, doing some minor changes on the domain of definition, 
we can prove the compactness of $\LL_\omega$. 
Recall that the compactness of an operator implies that its
spectrum is either finite or countable with $0$ on its 
closure (see for instance theorem III-6.26 of \cite{Kat66}). 

\subsection{Reducibility loss and quasi-periodic renormalization}
\label{chapter application} 

Given a map $F$ like (\ref{q.p. forced system interval}) with 
$f\in\BB$ and $\omega\in\T$ we 
denote by $f^{n}:\T \times \R \rightarrow \R$ the $x$-projection
of $F^n(x,\theta)$. Equivalently $f^n$ can be defined through the recurrence
\begin{equation}
\label{definicio falan}
f^n(\theta, x) = f( \theta +(n-1) \omega, f^{n-1}(\theta,x)). 
\end{equation}

From this point on, whenever  $\omega$ is used, it is assumed to 
be Diophantine.  Denote by $\Omega= \Omega_{\gamma,\tau}$
the set of Diophantine numbers,
that is the set of $\omega\in \T$
such that there exists $\gamma >0$ and $\tau \geq 1$
such that
\[
|q \omega - p| \geq \frac{\gamma}{|q|^{\tau}}, \quad 
\text{ for all } (p,q) \in \Z \times (\Z \setminus \{0\}). 
\]

Additionally, we will need to assume that the following conjecture is true.  

\begin{conj}
\label{conjecture H2} 
The operator $\TT_{\omega}$ (for any $\omega\in\Omega$) is an injective 
function when restricted to the domain $\BB\cap \DD(\TT)$. Moreover, 
there exist $U$ an open set of $\DD(\TT)$ containing $W^u(\Phi,\RR)
\cup W^s(\Phi,\RR)$\footnote{Here $W^s(\Phi,\RR)$ and $W^u(\Phi,\RR)$ 
are considered as the inclusion in $\BB$ of the 
stable and the unstable manifolds of the fixed point $\Phi$ (given by 
{\bf H0}) by the map $\RR$ in the topology of $\BB_0$.
}
where the operator $\TT_\omega$ is differentiable.
\end{conj}

In \cite{JRT11a} we discuss the difficulties for proving this 
conjecture, and in \cite{JRT11c} we show that the results obtained
assuming this conjecture are coherent with the numerical computations. 
Whenever the conjecture {\bf \ref{conjecture H2}} is needed for a 
result it is explicitly stated in the hypotheses. 

\subsubsection{Consequences for a two parametric family of maps} 
\label{Section consequences for a two parametric family of maps}

Consider a two parametric family of maps  $\{c(\alpha,\eps)\}_{(\alpha,\eps)\in A}$ 
contained in $\BB$, with $A = [a,b]\times[0,d]$ and $a$, $b$ 
and $d$ are real numbers (with $a<b$ and $0<d$). 
We assume that the dependency on the parameters is analytic.

Consider the following hypothesis on the family of maps. 
\begin{description}
\item[H1)]  The family $\{c(\alpha,\eps)\}_{(\alpha,\eps)\in A}$ uncouples for 
$\eps=0$, in the sense that the family $\{c(\alpha,0)\}_{\alpha\in[a,b]}$ 
does not depend on $\theta$ and it has a full cascade of 
period doubling bifurcations. We assume that the family 
$\{c(\alpha,0)\}_{\alpha\in[a,b]}$ crosses transversely the 
stable manifold of $\Phi$, the fixed point
of the renormalization operator, and each of 
the manifolds $\Sigma_n$ for any $n\geq 1$, where 
$\Sigma_{n}$ is the inclusion in $\BB$ of the 
set of one dimensional unimodal maps with 
a super-attracting $2^n$ periodic orbit.
\end{description}

In other words, we assume that the family $c(\alpha,\eps)$ 
can be written as, 
\[
c(\alpha,\eps)=c_0(\alpha) + \eps c_1(\alpha,\eps),
\]
with $\{c_0(\alpha)\}_{\alpha\in[a,b]}\subset \BB_0$  having a full 
cascade of period doubling bifurcations. 

Given a family $\{c(\alpha,\eps)\}_{(\alpha,\eps)\in A}$ satisfying 
the hypothesis {\bf H1},  let $\alpha_n$ be the parameter value for 
which the uncoupled family $\{c(\alpha,0)\}_{\alpha\in[a,b]}$
intersects the manifold $\Sigma_n$. Note that the critical 
point of the map $c(\alpha_n,0)$ is a $2^n$-periodic orbit.
Our main achievement in \cite{JRT11a} is 
to prove that from every 
parameter value $(\alpha_n,0)$ there are born two curves in 
the parameter space, each of them corresponding to a reducibility 
loss bifurcation. Now we introduce some technical definitions 
in order to give a more precise statement of this result.

Let $\RHH(\B_\rho,\W)$ denote the space of 
periodic real analytic maps from $\B_\rho$ to $\W$ and 
continuous in the closure of $\B_\rho$. 
Consider a map $f_0\in \BB$ and $\omega \in \Omega$, such that $f$ has a 
periodic invariant curve $x_0$ of rotation number $\omega$
with a  Lyapunov exponent bounded by certain $-K_0<0$. 
Using lemma 3.6 in \cite{JRT11a} we have that 
there exist a neighborhood $V\subset \BB$ of $f_0$ and a map 
$x\in \RHH(\B_\rho, \W)$ such that $x(f)$ is a periodic 
invariant curve of $f$  for any $f\in V$. 
Then we can define
the map $G_1 $ as
\begin{equation}
\label{equation definition G1}
\begin{array}{rccc}
G_1:& \Omega\times V &\rightarrow & \RHH(\B_\rho,\C) \\
\rule{0pt}{3ex} &  (\omega,g) & \mapsto &  
D_x g \big(\theta+\omega,g(\theta, \left[x(\omega,g)\right](\theta))\big) 
D_x g \big(\theta,\left[x(\omega,g)\right](\theta)\big). 
\end{array}
\end{equation}

On the other hand, we can consider the counterpart of the map 
$G_1$ in the uncoupled case. Given a map $f_0\in \BB_0$, consider 
$U\subset \BB_0$ a neighborhood of $f_0$ in the $\BB_0$ topology. 
Assume that $f_0$ has an attracting $2$-periodic 
orbit $x_0\in I$. Let $x=x(f)\in \W$ be the continuation of this 
periodic orbit for any $f\in U$.  We have  that $x$ depends 
analytically on the map, therefore it induces a map $x:U \rightarrow \W$. 
Then if we take $U$ 
small enough we have an analytic map 
$x:U\rightarrow \W$ such that $x[f]$ is a periodic orbit of period 2. 
Now we can consider the map
\begin{equation}
\label{equation definition widehatG1}
\begin{array}{rccc}
\widehat{G}_1:& U\subset \BB_0 &\rightarrow & \C \\
\rule{0pt}{3ex} &  f & \mapsto & D_x f \big(f(x[f])\big) D_x f \big( x[f] \big). 
\end{array}
\end{equation}

Note that $\widehat{G}_1$ corresponds to $G_1$ restricted to 
the space $\BB_0$ (but then $\widehat{G}_1(f)$ has to be 
seen as an element of $\RHH(\B_\rho,\W)$). 

Consider the sequences
\begin{equation}
\label{equation sequences corollary directions}
\begin{array}{rcll} 
\omega_k  & = & 2 \omega_{k-1},  & \text{ for }  k=1,..., n-1. \\ 
\rule{0ex}{4ex} 
f^{(n)}_k & = & \RR\left(f_{k-1}^{(n)}\right),
& \text{ for }  k=1,..., n-1. \\ 
\rule{0ex}{4ex} 
u^{(n)}_k & = & D \RR \left(f^{(n)}_{k-1}\right) u^{(n)}_{k-1},
&  \text{ for } k = 1, ..., n-1.\\
\rule{0ex}{4ex} 
v^{(n)}_k & = & D \TT_{\omega_{k-1}}  \left(f^{(n)}_{k-1}\right) v^{(n)}_{k-1},
&  \text{ for } k = 1, ..., n-1.
\end{array}
\end{equation}
with
\begin{equation}  
\label{equation sequences directions initial}
f^{(n)}_0  =  c(\alpha_n,0), \quad 
u^{(n)}_0  =  \partial_\alpha c(\alpha_n,0), \quad 
 v^{(n)}_0 =  \partial_\eps c(\alpha_n,0).  
\end{equation}

Note that $f^{(n)}_0$ tends to $W^s(\RR,\Phi)$ when $n$ grow. 
Then  $\{f^{(n)}_k\}_{0\leq k <n}$ attains to $W^s(\RR,\Phi) 
\cup W^u(\RR,\Phi)$ and consequently there 
exist $n_0$ s. t. $\{f^{(n)}_k\}_{0\leq k <n} \subset U$ , 
where $U$ is the neighborhood given in conjecture {\bf \ref{conjecture H2}}. 
If the conjecture is true, then the operator
 $\TT_\omega$ is differentiable in the orbit $\{f^{(n)}_k\}_{0\leq k <n} 
\subset U$.  

Consider the following 
hypothesis. 

\begin{description}
\item[H2)]  The family $\{c(\alpha,\eps)\}_{(\alpha,\eps)\in A}$ is such that 
\[
D G_1 \left(\omega_{n-1}, f^{(n)}_{n-1}\right) 
D\TT_{\omega_{n-2}}\left(f^{(n)}_{n-2}\right) \cdots 
D\TT_{\omega_0}\left(f^{(n)}_0\right) \partial_\eps c(\alpha_n,0),
\] 
has a unique non-degenerate minimum (respectively maximum) as a function from $\T$ to $\R$, 
for any $n\geq n_0$. 
\end{description}

Consider a family of maps $\{c(\alpha,\eps)\}_{(\alpha,\eps)\in A}$ 
such that the hypotheses {\bf H1} and {\bf H2} are satisfied
and  $\omega_0\in \Omega$. If the conjecture {\bf \ref{conjecture H2}} is 
true, then  theorem 3.8 in \cite{JRT11a} asserts that there exists 
$n_0$ such that, for any $n\geq n_0$, 
there exist two bifurcation curves around the 
parameter  value $(\alpha_n, 0)$,  such that they correspond to a
reducibility-loss bifurcation of the $2^n$-periodic invariant curve.
Moreover, these curves are locally expressed as 
$(\alpha_n + \alpha_n'(\omega) \eps + o(\eps),\eps)$ and 
$(\alpha_n^- + \beta_n'(\omega) \eps + o(\eps), \eps)$ with
\begin{equation}
\label{equation alpha n +}
\alpha'_n (\omega) = 
- \frac{m \left(
DG_1 \left(\omega_{n-1}, f^{(n)}_{n-1}\right) v_{n-1}^{(n)}
\right)
}{\rule{0ex}{3.5ex}  
D \widehat{G}_1 \left(f^{(n)}_{n-1}\right) u_{n-1}^{(n)}} ,
\end{equation}
and 
\begin{equation}
\label{equation alpha n -}
\beta'_n(\omega) = 
- \frac{M \left(
D G_1 \left(\omega_{n-1}, f^{(n)}_{n-1}\right) v_{n-1}^{(n)}
\right)
}{\rule{0ex}{3.5ex}  
D \widehat{G}_1 \left(f^{(n)}_{n-1}\right) u_{n-1}^{(n)}} ,
\end{equation}
where $G_1$ and $\widehat{G}_1$ are given by equations 
(\ref{equation definition G1}) and (\ref{equation definition widehatG1}), 
and $m$ and $M$ are the minimum and the maximum as operators, 
that is
\begin{equation}
\label{equation definition of minim} 
\begin{array}{rccc}
m:& \RHH(\B_\rho,\C) &\rightarrow & \R \\
\rule{0pt}{3ex} & g & \mapsto & \displaystyle \min_{\theta \in \T} g(\theta). 
\end{array}
\end{equation}
and
\begin{equation}
\label{equation definition of maximum} 
\begin{array}{rccc}
M:& \RHH(\B_\rho,\C) &\rightarrow & \R \\
\rule{0pt}{3ex} & g & \mapsto & \displaystyle \max_{\theta \in \T} g(\theta). 
\end{array}
\end{equation}

Now we can go back to the hypothesis {\bf H2}, which is not intuitive. 
Actually we can introduce a stronger condition which is much more easy to 
check. Moreover this condition is automatically satisfied by maps like 
the Forced Logistic Map. Consider a family of maps 
$\{c(\alpha,\eps)\}_{(\alpha,\eps)\in A}$ as before, 
satisfying hypothesis {\bf H1}. 

\begin{description}
\item[H2')]  The family $\{c(\alpha,\eps)\}_{(\alpha,\eps)\in A}$ is such that
the quasi-periodic perturbation $\partial_\eps c(\alpha,0)$ belongs to the set
$\BB_1$ (see equation (\ref{equation spaces bbk}))  for any value of 
$\alpha$ (with $(\alpha,0)\in A$). 
\end{description}

Proposition 3.10 in \cite{JRT11a} asserts that {\bf H2'} implies {\bf H2}.

\section{Universality for q.p. forced maps}
\label{section A not yet rigorous explanation} 

In \cite{JRT11p2} we have done a numerical study of the asymptotic 
behavior of the reducibility loss directions $\alpha_i'(\omega)$ 
of the FLM. This study is summarized in section \ref{sec:num obs}. 
Concretely we have done three different numerical observation 
on this asymptotic behavior, to which we refer as first, second 
and third numerical observations. On the other hand, formula 
(\ref{equation alpha n +}) provides an explicit expression for the reducibility 
loss directions $\alpha_i'(\omega)$ in terms of the quasi-periodic 
renormalization operator. In this section we propose  three different 
conjectures on the dynamics of the quasi-periodic renormalization 
operator which provide a suitable explanation 
to the numerical observations.

Due to the periodicity of the maps considered, the quasi-periodic 
renormalization has an intrinsic rotational symmetry. In section 
\ref{subsection rotational symmetry reduction} we reduce the symmetry 
by taken a suitable section, in a process analogous to a Poincar\'e section. 

In section \ref{subsection reduction to the dynamics} we reduce the 
problem to the dynamics of the q.p. renormalization operator. 
To do this it is necessary to introduce conjecture {\bf \ref{conjecture H3}}, 
in which we assume that the normal behavior of the 
operator for the iterates close to the stable and the unstable 
manifold is described by the linearization of the operator in 
the fixed point. 

Consider $\BB_1$ the space given by (\ref{equation spaces bbk}) for $k=1$.
In section \ref{subsection explanation 1st observation} we 
study the linearized dynamics of the 
renormalization operator but restricted to the space $\BB_1$. 
We use some symmetries of the map to perform some kind 
of ``Poincar\'e section'' of the operator.
Then we introduce conjecture {\bf \ref{conjecture H4}}, in which we require 
the ``Poincar\'e map'' to be contractive. Finally we present
theorem \ref{reduction alphas to dynamics of renormalization} 
which gives a theoretical explanation to the first numerical 
observation described in section \ref{section intro}. 

In section \ref{subsection explanation 2nd observation} we prove that, 
under appropriate hypotheses, the behavior associated to the first 
numerical observation implies the behavior associated to the second 
observation. In this section we introduce conjecture {\bf \ref{conjecture H5}},
which is necessary to check that the appropriate hypotheses are 
satisfied in the case of the Forced Logistic Map. 

In section \ref{subsection explanation 3rd observation} we analyze 
what happens when a map does not satisfy hypothesis {\bf H2'}, as it 
happened in sections \ref{subsection explanation 1st observation} and 
\ref{subsection explanation 2nd observation}. This analysis 
provides an explanation to the third numerical observation. 

All proofs have been moved to the end of their respective subsections to 
make the presentation clearer.  

\subsection{Rotational symmetry reduction} 
\label{subsection  rotational symmetry reduction}

Given a function $g:\T \times I_\delta \rightarrow I_\delta$ 
in $\BB$ we can consider the function $\tilde{g}$ defined as 
$\tilde{g}(\theta,x) = g (\theta+\gamma, x)$ for some $\gamma\in\T$. 
Maps like (\ref{q.p. forced system interval}) determined 
by $f=g$ or by $f=\tilde{g}$ exhibit essentially the same dynamics, 
although (from the functional point of view) they are not the same map. 
For example they have different Fourier expansion. 
Roughly speaking, this fact induce a rotational symmetry 
on the derivative of the quasi-periodic renormalization 
operator $\TT_{\omega}$. To follow with our study we need to 
remove this symmetry from the problem. 

Given $\gamma\in \T$,  consider the following auxiliary function 
\begin{equation}
\label{definition tgamma}
\begin{array}{rccc}
t_\gamma:&  \BB &\rightarrow &  \BB  \\
\rule{0pt}{5ex} & v(\theta, z) & \mapsto & v(\theta + \gamma, z). 
\end{array}
\end{equation}

Let $\BB_1$ be the subspace of $\BB$ defined by (\ref{equation spaces bbk}) 
for $k=1$. The space $\BB_1$ is indeed the image of the 
projection $\pi_1:\BB \rightarrow \BB$ defined as 
\begin{equation}
\label{projection pi_1}
\left[\pi_1 (v)\right](\theta , x)  =  
\left( \int_0^1 v(\theta,x) \cos(2\pi x) d\theta \right) \cos (2\pi\theta) + 
\left( \int_0^1 v(\theta,x) \sin(2\pi x) d\theta \right) \sin (2\pi\theta). 
\end{equation}

Given  $x_0\in \W \cap \R $ and $\theta_0\in\T$ we can 
also consider the sets
\[\BB_1' = \BB_1' (\theta_0,x_0) = 
\{f\in \BB_1 \thinspace | f(\theta_0,x_0)=0, \partial_\theta f(\theta_0,x_0)>0 \},\]
and 
\[\BB' = \BB' (\theta_0,x_0) = 
\{f\in \BB \thinspace | \pi_1(f) \in \BB_1' \}. 
\]

\begin{prop}
\label{proposion projections of BB1p}
For a fixed $x_0\in \W \cap \R$ and $\theta_0\in \T$, we have that 
$\BB_1' (\theta_0, x_0)$ is an open subset of a codimension 
one linear subspace of $\BB_1$. Moreover for any $v\in \BB_1 \setminus\{0\}$ 
there exists  a unique $\gamma_0\in \T$ such that $t_{\gamma_0}(v)\in \BB_1' 
(\theta_0,x_0)$. Therefore for any $v\in \BB$ such that $\pi_1(v) \in \BB_1 
\setminus\{0\}$ there exists  a unique $\gamma_0\in \T$ such that 
$t_{\gamma_0}(v)\in \BB'(\theta_0,x_0)$.
\end{prop}

Consider a two parametric family of maps $\{c(\alpha,\eps)\}_{(\alpha,\eps)
\in A}$ contained in $\BB$ satisfying the hypotheses {\bf H1} and {\bf H2}
as in section \ref{chapter application}. Consider also the reducibility 
loss bifurcation curves  associated to the $2^n$-periodic orbit with 
slopes given by (\ref{equation alpha n +}) and (\ref{equation alpha n -}).
The goal of this section is to use proposition 
\ref{proposion projections of BB1p} to express formulas  
(\ref{equation alpha n +}) and (\ref{equation alpha n -}) in 
terms of vectors in $\BB_1' (\theta_0, x_0)$. The case $\beta'_n(\omega)$ is
omitted from now on in the discussion since it is completely
analogous to the case considered here, one only has to replace the
appearances of a minimum by a maximum.

Consider the sequences $\{\omega_k\}$, $\{f^{(n)}_k\}$ and $\{u^{(n)}_k\}$ given by 
(\ref{equation sequences corollary directions}) and 
(\ref{equation sequences directions initial}). Consider now 
the sequence 
\begin{equation} 
\label{equation v_k in section} 
\tilde{v}^{(n)}_k = t_{\gamma\left(\tilde{v}^{(n)}_{k-1}\right)} 
\left( D \TT_{\omega_{k-1}}  \left(f^{(n)}_{k-1}\right) \tilde{v}^{(n)}_{k-1} \right)
\text{ for } k = 1, ..., n-1, 
\end{equation} 
and 
\begin{equation}
\label{equation v_0 in section} 
v^{(n)}_0 =  t_{\gamma_0} \left(\partial_\eps c(\alpha_n,0)\right), 
\end{equation} 
where $\gamma(\tilde{v}^{(n)}_{k-1})$ and $\gamma_0$ are chosen such 
that $\tilde{v}^{(n)}_{k}$ belongs 
to $\BB' (\theta_0, x_0)$ for $k=0,1, ..., n$. 
If the projection of $D\TT_{\omega_{k-1}}  \left(f^{(n)}_{k-1}
\right) \tilde{v}^{(n)}_{k-1}$ in $\BB_1$ is non zero, then 
$\gamma\left(\tilde{v}^{(n)}_{k-1}\right)$ is uniquely determined and 
the vectors $\tilde{v}^{(n)}_k$ are well defined.

\begin{thm} 
\label{thm symmetry reduction} 
Consider a family of maps $\{c(\alpha,\eps)\}_{(\alpha,\eps)\in A}$ 
such that the hypotheses {\bf H1} and {\bf H2} are satisfied. Assume 
also that $\omega_0\in \Omega$ and that the conjecture {\bf \ref{conjecture H2}} is 
true. Let $\{\omega_k\}$, $\{f^{(n)}_k\}$ and $\{u^{(n)}_k\}$ be defined by
(\ref{equation sequences corollary directions}) and
(\ref{equation sequences directions initial}) and 
$\tilde{v}^{(n)}_k$ be defined by (\ref{equation v_k in section}) 
and (\ref{equation v_0 in section}).  Assume also that 
the projection of $D\TT_{\omega_{k-1}}  \left(f^{(n)}_{k-1}
\right) \tilde{v}^{(n)}_{k-1}$ in $\BB_1$ (given by (\ref{projection pi_1})) 
is non zero.

Then the slopes $\alpha_n'$ of the reducibility loss bifurcations 
given by (\ref{equation alpha n +}) can be also written as
\begin{equation}
\label{equation alpha n + section}
\alpha'_n (\omega) = 
- \frac{m \left(
DG_1 \left(\omega_{n-1}, f^{(n)}_{n-1}\right) \tilde{v}_{n-1}^{(n)}
\right)
}{\rule{0ex}{3.5ex}  
D \widehat{G}_1 \left(f^{(n)}_{n-1}\right) u_{n-1}^{(n)}} ,
\end{equation}
where $G_1$,  $\widehat{G}_1$ and $m$ 
are given by equations (\ref{equation definition G1}),
(\ref{equation definition widehatG1}) and (\ref{equation definition of minim}). 
\end{thm}

\subsubsection{Proofs} 

\begin{lem}
\label{lemma properties of tgamma}
Consider the function $t_\gamma$ given by (\ref{definition 
tgamma}). 
\begin{enumerate}
\item For any  $f\in \BB$ and  $\gamma_1 , \gamma_2 \in \T$,
\[t_{\gamma_1+\gamma_2} (v) = t_{\gamma_1} \circ \ t_{\gamma_2} (v).\]

\item For any $f\in \BB$ and $\gamma \in \T$ we have  $\|t_\gamma (f) \| = \|f\|$ 
(recall that the norm of $\BB$ considered is 
the supremum norm in $\B_\rho \times\W$).

\item Let $\TT_\omega:  \DD(\TT) \rightarrow \BB$ be the renormalization operator, and
$\Phi$ a fixed point. Then we have that $t_\gamma$ and the differential of 
$\TT_\omega$ in
the fixed point commute. In other words, we have
\[
\left[ t_\gamma \circ D\TT_\omega(\Phi)  \right](v) = 
\left[ D\TT_\omega(\Phi)\right](t_\gamma(v)),
\]
for any $v\in \BB$ and $\gamma \in \T$.
\end{enumerate}
\end{lem}

\begin{proof}
The first point follows easily since 
\[\left[t_{\gamma_1+\gamma_2}(v)\right](\theta,z) 
= v(\theta +\gamma_1+\gamma_2,z)
= \left[t_{\gamma_2}(v)\right](\theta+\gamma_1,z)= 
 \left[t_{\gamma_1}\left(t_{\gamma_2}(v)\right)\right](\theta,z).\]

For the second point of the proposition, recall that the norm considered 
in $\BB$ is the supremum norm in the set $\B_{\rho}\times\W$.
Using the invariance of this set by a translation on the first variable 
we have have
\[
\|v\| =\sup_{\B_{\rho}\times\W } | v(\theta,x)| = 
\sup_{\B_{\rho}\times\W} |v(\theta+\gamma,x)| =
\|t_\gamma(v)\|. 
\]

Let us focus now in the third point of the proposition. Given $v\in \BB$ 
consider its complex Fourier expansion on the $\theta$ variable. 
\[ 
v(\theta,z)= \sum_{k\in \Z} c_k(z) e^{2\pi k\theta i}. 
\]
Then we have that the complex Fourier expansion of the map $t_\gamma(v)$ 
is given by
\[ 
\left[t_\gamma(v)\right](\theta,z)= 
\sum_{k\in \Z} \left(e^{2\pi k\gamma i} c_k(z) \right)e^{2\pi k\theta i }. 
\]
Using this, the expansion of $D\TT_\omega(\Phi)v$ given by equation 
(\ref{equation Fourier expansion differential renormalization operator})
 and 
the linearity of the operators $L_1$ and $L_2$ we have 
\begin{eqnarray}
\left[D\TT_\omega(\Phi)t_\gamma(v)\right](\theta,z) & =& 
D\RR_\delta (c_0)(z) +  \sum_{k\in \Z\setminus\{0\}} 
\left[L_1\left(e^{2\pi k\gamma i} c_k(z) \right) + 
L_2\left(e^{2\pi k\gamma i} c_k(z) \right) e^{2\pi k\omega i } \right] 
e^{2\pi k\theta i }  \nonumber \\
&= &\rule{0ex}{4ex}  D\RR_\delta (c_0)(z) +  \sum_{k\in \Z\setminus\{0\}} 
\left[L_1\left(c_k(z) \right) + 
L_2\left( c_k(z) \right) e^{2\pi k\omega i } \right] 
e^{2\pi k(\theta+\gamma) i }  \nonumber \\
& = &\rule{0ex}{4ex}
\left[t_\gamma\left(D\TT_\omega(\Phi)(v)\right)\right](\theta,z).
\nonumber
\end{eqnarray}
\end{proof}


\begin{proof}[Proof of proposition 
\ref{proposion projections of BB1p}]
Consider the map $\Ev:\BB_1\rightarrow \R$ the 
evaluation map defined as $\Ev(v)=v(\theta_0,x_0)$. Note 
that the evaluation of a map in a given point is differentiable 
as a function (see proposition 2.4.17 in \cite{AMR88}). Then we have 
that set $\BB_1'(\theta_0,x_0)$ is 
an open subset of the set $\Ev^{-1}(0)$, which is 
a codimension one Banach space. We also have that $\Ev^{-1}(0)$ is a 
linear subspace because  $\Ev(\cdot)$ is a linear function. 

Let us focus now on the second part of the proposition. 
Given $v\in \BB_1$ we have that 
\[
v(\theta,z) = A(z)\cos(2\pi\theta) + B(z) \sin(2\pi \theta).
\] with $A$ and $B$ in $\RHH(\W)$. We have that 
\[[t_\gamma(v)](\theta_0,x_0)= A(x_0)\cos(2\pi(\theta_0+\gamma))
+ B(x_0) \sin(2\pi(\theta_0+\gamma))= \tilde{A} \cos(2\pi\gamma) 
+ \tilde{B} \sin(2\pi \gamma) ,\] 
with $\tilde{A} = A(x_0)\cos(2\pi\theta_0) +
B(x_0)\sin(2\pi\theta_0)$ and $\tilde{B}= B(x_0)\cos(2\pi \theta_0) 
- A(x_0)\sin(2\pi \theta_0)$. Then taking $\gamma_0 =
\frac{1}{2\pi}\arctan\left(-\frac{\tilde{A}}{\tilde{B}}\right)$ and
$\gamma_1=\gamma_0+\frac{1}{2}$ we have  that $t_{\gamma_0}(v)$ and 
$t_{\gamma_1}(v)$ belong to $\Ev^{-1}(0)$ but only one of them belongs 
to $\BB_1'(\theta_0,x_0)$. 
\end{proof}

\begin{proof}[Proof of theorem \ref{thm symmetry reduction}]
We will need the following lemma for the proof. 
\begin{lem} 
\label{lemma theorem proof of the universality conjecture} 
Consider the function $t_\gamma$ defined in 
equation (\ref{definition tgamma}) and   
the set $\Sigma_1$ of one dimensional 
unimodal maps such that its critical point is a two periodic orbit. 
Then for any $\omega\in \Omega$, $f\in\Sigma_1$ and $v\in\BB$ we have
that
\[m \left(D G_1 \left(\omega, f \right) v \right) = 
m \left(D G_1 \left(\omega, f \right) t_\gamma(v)\right),\]
for any $\gamma\in \T$. 
\end{lem} 

\begin{proof}
Let $\tilde{t}_\gamma:C(\T,\R) \rightarrow C(\T,\R)$ be the 
operator defined as $[\tilde{t}_\gamma(p)](\theta) =p(\theta +\gamma)$. 
Note that for any function $p:\T \rightarrow \R$  and $\gamma \in \T$ 
we have that 
\[m(p)= \min_{\theta\in\T} p(\theta) = \min_{\theta\in\T} p(\theta+ \gamma) = 
\min_{\theta\in\T}[\tilde{t}_\gamma (p)](\theta)= m\left(\tilde{t}_\gamma (p)\right).\] 
Hence, 
\[
m \left(D G_1 \left(\omega, f \right) v \right) = 
 \min_{\theta \in \T} \left[
\tilde{t}_\gamma \left( D G_1 \left(\omega, f \right) v \right)  \right].
\]

Since $f\in\Sigma_1$ we have that 
$D G_1 \left(\omega, f \right)$ is explicitly given by 
proposition 3.12 in \cite{JRT11a}. Using 
this it is easy to check that 
$\tilde{t}_\gamma \circ D G_1 \left(\omega, f \right) 
v= D G_1 \left(\omega, f \right) t_\gamma(v)$. Applying 
this to the equation above, the result follows. 
\end{proof}

Using lemma \ref{lemma theorem proof of the universality conjecture} 
we have 
\[
m \left( D G_1 \left( \omega_k, f^*_{1} \right) , v_k \right) =  
m \left( D G_1 \left( \omega_k, f^*_{1} \right) ,  t_\gamma \left(v_k \right) \right),
\]
with $\gamma$ any value in $\T$. 
Since the value $\gamma$ is arbitrary, we can 
choose $\gamma= \gamma_k + \tilde{\gamma}$ with $\gamma_k$ and 
$\tilde{\gamma}$ any values in $\T$.
Recall now that the values $\omega_n$ and $v_n$ 
are defined by  the recurrence (\ref{equation sequences directions simplified}). 
Using these recurrences and the first and third properties of 
proposition \ref{lemma properties of tgamma} we have that 
\[
t_\gamma(v_k) = t_{\gamma_{k}}\left(t_{\tilde{\gamma}}
\left(D\TT_{\omega_{k-1}} (\Phi) v_{k-1}\right) \right) = 
t_{\gamma_{k}}\left(D\TT_{\omega_{k-1}} (\Phi) t_{\tilde{\gamma}}
\left(v_{k-1}\right) \right).
\]
This can be reproduced at every step of the recurrence in such a way 
that the sequence $v_k$ for $k=0,\dots , n$ can be replaced 
by the sequence $t_{\gamma_k}(v_k)$ without loss of generality. 

By hypothesis we have that the projection of $\TT_{\omega_{k-1}}  \left(f^{(n)}_{k-1}
\right) \tilde{v}^{(n)}_{k-1}$ in $\BB_1$ is non zero. We can 
apply proposition \ref{proposion projections of BB1p}, 
then the values of $\gamma_k$ can be chosen in such a way that 
$t_{\gamma_k}(\TT_{\omega_{k-1}}  \left(f^{(n)}_{k-1}
\right) \tilde{v}^{(n)}_{k-1})$ belongs to $\BB_1'$ for any $k\geq0$.
\end{proof} 

\subsection{Reduction to the dynamics of the renormalization operator}
\label{subsection reduction to the dynamics}

Consider a two parametric family of maps $\{c(\alpha,\eps)\}_{(\alpha,\eps)
\in A}$ contained in $\BB$ satisfying the hypotheses {\bf H1} and {\bf H2}
as in section \ref{chapter application}. Consider also the reducibility
loss bifurcation curves  associated to the $2^n$-periodic orbit with
slopes $\alpha'_n(\omega)$ and $\beta'_n(\omega)$ given by 
(\ref{equation alpha n +}) and (\ref{equation alpha n -}).
As in section \ref{subsection  rotational symmetry reduction} we 
omit the case concerning $\beta'_n(\omega)$ since it is completely 
analogous to the case concerning $\alpha'_n(\omega)$. The goal of this section 
is to reduce the problem of describing the asymptotic behavior 
of ${\alpha'_n(\omega_0,c_1)}/{\alpha'_{n-1}(\omega_0,c_1)}$ to 
the dynamics of the quasi-periodic renormalization operator. 

\begin{defin}
\label{definition equivalence of banach sequences}
Given two sequences $\{r_i\}_{i\in \Z_+}$ and $\{s_i\}_{i\in \Z_+}$ in 
a Banach space, we will say that they are {\bf asymptotically
equivalent} if there exists $0<\rho <1$ and $k_0$ such that
\[
\| r_i - s_i\| \leq k_0 \rho^i \quad \forall i\in \Z_+ .
\]
We will commit an abuse of notation and 
denote this equivalence relation by $s_i \sim r_i$ instead of 
$\{r_i\}_{i\in \Z_+} \sim \{s_i\}_{i\in \Z_+}$.
\end{defin}

Let us remark that it should be more precise to speak about geometric
asymptotic equivalence, but the word geometrically has been 
omitted for simplicity.

Given a family $\{c(\alpha,\eps)\}_{(\alpha,\eps)\in A}$ 
satisfying the hypotheses  {\bf H1} and {\bf H2} as before
and a
fixed Diophantine rotation number $\omega_0$, consider 
$\omega_k$ , $f^{(n)}_k$, $u^{(n)}_k$
given by
(\ref{equation sequences corollary directions}) and 
$\tilde{v}^{(n)}_k$ given by (\ref{equation v_k in section}), 
with $ f^{(n)}_0$ and $u^{(n)}_0$ given by (\ref
{equation sequences directions initial}) and 
$v^{(n)}_0$ given by (\ref{equation v_0 in section}). 
Note that $f^{(n)}_k$, $u^{(n)}_k$ and $v^{(n)}_k$ depend on $c$,
the family of maps considered, and  the vectors $v^{(n)}_k$ depend also 
on the initial value of the rotation number $\omega_0$. 
In general this dependence will be omitted to keep the notation simple. 
If two different families or two different values of the rotation number should 
be considered then we will make the dependence explicit.

Let  $\alpha^*$ denote the parameter value such that the
family $\{c(\alpha,0)\}_{(\alpha,0)\in A}$ intersects
with $W^s(\Phi, \RR)$ and $f_j^*$ denote the intersection 
of $W^u(\Phi, \RR)$ with the manifold $\Sigma_j$. Consider then
\begin{equation}
\label{equation sequences directions simplified}
\begin{array}{rcl} 
\rule{0ex}{2.5ex}\omega_k & = & 2 \omega_{k-1}, \text{ for }  k = 1, ..., n-1, \\
\rule{0ex}{4.5ex} 
u_k & = &\displaystyle  \left\{ \begin{array}{ll} 
\rule{0ex}{2.5ex} \displaystyle   D \RR \left(\Phi \right) u_{k-1},  
      &  \text{ for } k = 1, \dots, [n/2]-1, \\ 
\rule{0ex}{2.5ex} \displaystyle   D \RR \left(f_{n-k}^* \right) u_{k-1},  
      &  \text{ for } k = [n/2], \dots, n-1. 
\end{array} \right.  \\
\\\rule{0ex}{4.5ex} 
v_k & = & \displaystyle \left\{ \begin{array}{ll} 
\rule{0ex}{2.5ex} \displaystyle 
t_{\gamma\left(\tilde{v}_{k-1}\right)}
\left( D\TT_{\omega_{k-1}} \left(\Phi \right) v_{k-1} \right),
   & \text{ for } k = 1, ..., [n/2]-1,\\
\rule{0ex}{2.5ex} \displaystyle  t_{\gamma\left(\tilde{v}_{k-1}\right)}
\left( D\TT_{\omega_{k-1}} \left(f_{n-k}^*\right) v_{k-1}\right), 
   &\text{ for } k=[n/2],\dots, n-1.
\end{array} \right. \\
\end{array}
\end{equation}
with 
\[  u_0  =  \partial_\alpha c(\alpha^*,0), \quad 
 v_0 =  t_{\gamma_0} \left( \partial_\eps c(\alpha^*,0) \right),   \]
and $\gamma(\tilde{v}_{s-1})$ and $\gamma_0$ are chosen such
that $\tilde{v}^{(n)}_{s}$ belongs
to $\BB_1' (\theta_0, x_0)$ for any $s=1, ..., n$.

\begin{conj} 
\label{conjecture H3}
For any family of maps $\{c(\alpha,\eps)\}_{(\alpha,\eps)\in A}$ 
satisfying {\bf H1 } and {\bf H2}, 
assume that 
\[
\frac{\tv^{(n)}_{n-1}}{\|\tv^{(n)}_{n-1} \|} \sim \frac{v_{n-1}}{\|v_{n-1} \|},
\]
with $\tv^{(n)}_{n-1}$ and $v_{n-1}$ given by 
(\ref{equation v_k in section}) and 
(\ref{equation sequences directions simplified}). 
 Also assume that there exists a constant $C>0$ such that 
\[
\|v_{n-1}\| > C \text{ for any }  n > 0 .  
\]
Finally assume  that there exists a constant $C_0>0$ such that 
\[\left|m \left( DG_1\left(\omega_{n-1},f^*_1,
\frac{v_{n-1}}{\|  v_{n-1} \|}\right) \right) \right|> C_0,\]
for any $n\geq0$ and $\omega_0$ Diophantine, where 
$m$ is given by (\ref{equation definition of minim}), 
$G_1$ by (\ref{equation definition G1}) and $\{f_1^*\}= 
W^u(\RR,\Phi)\cap\Sigma_1 $. 
\end{conj}

\begin{figure}[t]
\centering
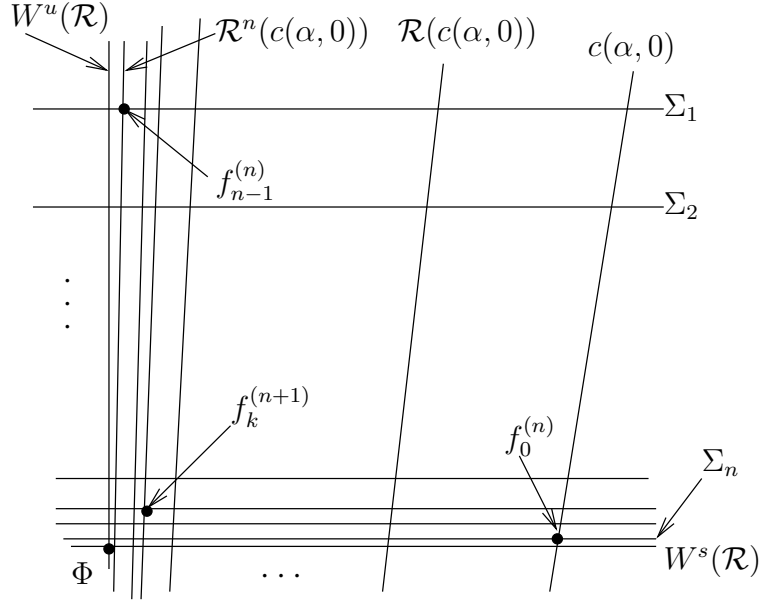
\caption{Representation of the dynamics of $\RR$ around 
its fixed point $\Phi$, see the text for more details. 
}
\label{Esquema Varietats}
\end{figure}

In other words we assume that the asymptotic behavior of the 
vectors  $\tv^{(n)}_{n-1}$ is determined by the 
linearization of the renormalization operator in the 
fixed point. Moreover we assume that the modulus of the 
vector does not decrease to zero. In figure \ref{Esquema Varietats} we 
have a schematic representation of the orbit $f^{(n)}_k$ with respect 
to the fixed point $\Phi$ and its stable and unstable manifolds $W^s(\Phi,\RR)$ and 
$W^u(\Phi,\RR)$. We have that the orbit of $f_0^{(n)}$ corresponds 
to a passage near a saddle point. Note that the initial point $f_0^{(n)}$ is always
in $\{c(\alpha,0)\}_{(\alpha,0) \in A}$, the final point $f_{n-1}^{(n)}$ is 
always in $\Sigma_1$ for any $n$, and the orbit of the points spends more 
and more iterates in a neighborhood of $\Phi$ when $n$ is increased. 

To justify conjecture \ref{conjecture H3}, let us remark that the initial point 
$f_0^{(n)}=c(\alpha_n,0)$ corresponds to the family $c(\alpha,0)$ intersected
with the manifold $\Sigma_n$. On the other hand, the point $f_{n-1}^{(n)}$
corresponds to the intersection of $\RR^n (c(\alpha,0))$ with $\Sigma_1$. When
$n$ is increased we have that $f_0^{(n)}$ converges to the intersection
of $c(\alpha,0)$ with the stable manifold $W^s(\Phi,\RR)$ and 
$f_{n-1}^{(n)}$ converges to 
the intersection or $\Sigma_1$ with the unstable manifold $W^u(\Phi,\RR)$. 
Moreover when $n$ is increased the intermediate points  $f_k^{(n)}$  
spend more and more iterates in an arbitrarily small neighborhood of 
$\Phi$. Actually, this is the typical behavior of passages close to a 
saddle fixed point. Then we can expect that the asymptotic behavior of 
the vectors $\tv^{(n)}_{n-1}$ is determined by the dynamics of the fixed point.
The  last part of the conjecture  can be understood as a kind of uniform 
transversality of the vectors $\frac{v_{n-1}}{\|v_{n-1}\|}$ with respect 
to the manifold defined by the zeros of this function. 
This conjecture is checked numerically for the case of the Forced 
Logistic Map in \cite{JRT11c}. 

Finally we will need the following extension of the hypothesis {\bf H2} 

\begin{description}
\item[H3)]   Consider a two parametric family of maps 
$\{c(\alpha,\eps)\}_{(\alpha,\eps)\in A}$ (with $A\subset\R^2$)
satisfying
{\bf H1} and {\bf H2} and a fixed Diophantine rotation 
number $\omega_0$. Consider also $\omega_n$ and $v_n$ given 
by (\ref{equation sequences directions simplified}) 
and the point  $\{f^*_{1}\} = W^u(\RR,\Phi)\cap \Sigma_1$.  
We assume that $DG_1(\omega_{n-1}, f^*_1 )
 v^{(n)}_{n-1}$ has a unique non-degenerate minimum for 
any $\omega_0\in \Omega$ and $n\geq 0$. Assume also that
the projection of $D\TT_{\omega_{k-1}}  \left(f^{(n)}_{k-1}
\right) \tilde{v}^{(n)}_{k-1}$ in $\BB_1$ given by (\ref{projection pi_1})
is non zero.
\end{description}

Using the notation and the hypotheses introduced so far, we have the following 
results on the asymptotic behavior of the quotients
$\frac{\alpha_n'(c,\omega_0)}{ \alpha_{n-1}'(c,\omega_0)}$.

\begin{thm}
\label{reduction alphas to dynamics of renormalization}
Let $\{c(\alpha,\eps)\}_{(\alpha,\eps)\in A}$ (with $A\subset\R^2$) be 
a two parametric family of q.p. forced maps 
satisfying {\bf H1}, {\bf H2} and {\bf H3}. Suppose that 
$\omega_0 $ is Diophantine ($\omega_0 \in \Omega$). 
Consider the loss of reducibility 
directions $\alpha_n'(c,\omega_0)$ and the sequences $u_n$ and $v_n$ given by 
(\ref{equation sequences directions simplified}). Additionally assume that 
conjectures {\bf \ref{conjecture H2}} and {\bf \ref{conjecture H3}} are true. Then we have that 
\begin{equation}
\label{equation theorem reduction to dynamics of renormalization}
\frac{\alpha_n'(c,\omega_0)}{ \alpha_{n-1}'(c,\omega_0)} \sim
\deltabf^{-1} \cdot \frac{
\displaystyle m \left( DG_1\left(\omega_{n-1}, f^*_{1} , 
\frac{v_{n-1}}{\|v_{n-1}\|} \right) \right)
}{ \rule{0pt}{4ex} \displaystyle 
m \left( DG_1\left(\omega_{n-2},  f^*_1,  \frac{v_{n-2}}
{\|v_{n-2}\|} \right) \right)
}\cdot 
\left\| D\TT_{\omega_{n-2}}(f_2^*) \frac{ v_{n-2}}{\|v_{n-2}\|}
\right\|, 
\end{equation} 
where $m$ is given by (\ref{equation definition of minim}),
$G_1$ by (\ref{equation definition G1}), $\{f_1^*\}= 
W^u(\RR,\Phi)\cap\Sigma_1$ is the intersection of the unstable 
manifold of $\RR$ at the fixed point $\Phi$ with the manifold 
$\Sigma_1$ and $\deltabf$ is the universal Feigenbaum constant.
\end{thm}

The interpretation of this result, which will become clearer in section 
\ref{subsection explanation 1st observation}, is the following. Let  $c_1$ and $c_2$ be
two families of q.p. forced maps satisfying {\bf H1} ,  {\bf H2} and
{\bf H3} and $\omega_0$ a Diophantine number. Consider the loss of
reducibility directions $\alpha_n'(c_i,\omega_0)$ associated to each family of
maps, as well as the sequences $v_n(c_i,\omega_0)$ given by the
recurrence (\ref{equation sequences directions simplified}) with
$v_0(c_i,\omega_0)=\partial_{\eps} c_i (\alpha^*,0)$. Then, to show that 
\[\frac{\alpha_n'(\omega_0,c_1)}{ \alpha_{n-1}'(\omega_0,c_1)}
\sim
\frac{\alpha_n'(\omega_0,c_2)}{ \alpha_{n-1}'(\omega_0,c_2)},\]
it is enough prove that
\[\frac{ v_{k}(\omega_0,c_1)}{\|v_{k}(\omega_0,c_1)\|} 
\sim \frac{ v_{k}(\omega_0,c_2)}{\|v_{k}(\omega_0,c_2)\|}. \]

\subsubsection{Proofs} 
To prove theorem \ref{reduction alphas to dynamics of renormalization}
it is necessary to introduce the 
following technical lemmas on the equivalence relation $\sim$.

\begin{lem}
\label{lemma equivalencia multiplicacio}
Given four different sequences $\{r_i^{(1)}\}$, $\{r_i^{(2)}\}$, 
$\{s_i^{(2)}\}$ and $\{s_i^{(2)}\}$ all of them in $\ell^\infty(\R)$, assume that 
\begin{equation}
\label{hypothesi lemma equivalencia multiplicacio}
r_i^{(1)} \sim r_i^{(2)} \text{ and } s_i^{(1)} \sim s_i^{(2)}. 
\end{equation}
Then we have that 
\[
r_i^{(1)}s_i^{(1)}  \sim r_i^{(2)} s_i^{(2)}.
\]
\end{lem}
\begin{proof}
We have that 
\begin{eqnarray}
|r_i^{(1)}s_i^{(1)} - r_i^{(2)} s_i^{(2)}|  & = & 
|r_i^{(1)}(s_i^{(1)} - s_i^{(2)})  +   s_i^{(2)}(r_i^{(1)} - r_i^{(2)})| \nonumber
\\ 
& \leq & |r_i^{(1)}| |s_i^{(1)} - s_i^{(2)}|  +   |s_i^{(2)}| |r_i^{(1)} - r_i^{(2)}|. 
\nonumber
\end{eqnarray}
From $\{r_i^{(1)}\} \in \ell^\infty(\R)$ and $\{s_i^{(2)}\} \in \ell^\infty(\R)$ it follows 
that there exist a constant $K_0$ such that $|r_i^{(1)}| < K_0$ and  
$|s_i^{(2)}| < K_0$. Using 
(\ref{hypothesi lemma equivalencia multiplicacio}) and the bound above 
the lemma follows easily. 
\end{proof}

\begin{lem}
\label{lemma equivalencia divisio} 
Let $\{r_i\}$ and $\{s_i\}$ be two different sequences of real numbers 
with $r_i\sim s_i$ and $s_i> C_0$ for any $i\geq n_0$. Then we have 
\[
\frac{r_i}{s_i} \sim 1.
\]
\end{lem}
\begin{proof}
If follows easily from  $r_i\sim s_i$ and the following bound
\begin{equation}
\left|\frac{r_i}{s_i} - 1\right| =  \left|\frac{1}{s_i}\right| |r_i - s_i|  
\leq   \frac{1}{C_0} |r_i - s_i|.  \nonumber
\end{equation}
\end{proof}

\begin{lem}
\label{lemma equivalencia funcio diferenciable}
Let $B$ be a Banach space and $N$ a normed space. Consider that we 
have $\{f_n\}_{n\geq0}$ a sequence on $B$ such that $f_n \sim f$, with 
$f\in B$. Also consider  $\{u_n\}_{n\geq0}$ a sequence of vectors 
on $N$ and a function $G:B \times N \rightarrow \R$. Assume that 
$G$ is differentiable w.r.t the first variable in a neighborhood $V$ of 
$f$ and 
\begin{equation}
\label{uniform bound of the partial}
\left\|\frac{\partial}{\partial x_1} G(g,u_n) v\right\| \leq  C \| v\|, 
\end{equation}
for any $n\geq n_0$, $g \in V$ and $v\in B$.  

Then we have that 
\[
G(f_n,u_n) \sim G(f,u_n).
\]
\end{lem} 
\begin{proof}
From $f_n \sim f$ we have that $f_n$ tends to $f$ with a geometric rate. In other
words, 
we have that 
\[f_n = f + \Delta_n, \text{ with } \Delta_n \in B \text{ and } 
\|\Delta_n \| < k_0 \rho^n,\] 
with $k_0>0$ independent of $n$ and $\rho <1$. 

In particular we have that $f_n$ belongs to a neighborhood 
$V$ of $f$ for any $n\geq n_0$. We can consider the auxiliary functions
$H_n:[0,1] \rightarrow \R$ given as $H_n(t)= G(f+ t \Delta_n, u_n)$. If we 
apply the mean value theorem to $H_n$ we have that there exist 
a real value $r_n\in(0,1)$ such that 
\[
G(f_n, u_n) = G(f, u_n) +  \frac{\partial}{\partial x_1} G(f + r_n \Delta_n,u_n) \Delta_n. 
\]
Remark that for any $n\geq n_0$ we have that $f + r_n \Delta_n$ belongs to the 
neighborhood $V$ of $f$.
Therefore we can apply the bound (\ref{uniform bound of the partial}) given by hypothesis, 
then we have that 
\begin{eqnarray}
| G(f_n,u_n) - G(f,u_n)| &  =  & |G(f + \Delta_n, u_n) - G(f,u_n) |  \nonumber \\
& =  & | \frac{\partial}{\partial x_1} G(f + r_n \Delta_n,u_n) \Delta_n | 
 \leq C \|\Delta_n \| \leq   K_0 \rho^n, \nonumber 
\end{eqnarray}
for any $n\geq n_0$. 
\end{proof}

\begin{proof}[Proof of theorem \ref{reduction alphas to dynamics of renormalization}]
To simplify the expression
of $\alpha'_n(\omega_0,c)$ in terms of the (q.p. forced) renormalization 
operator let us consider the following functions,
\begin{equation}
\label{equation definition L}
\begin{array}{rccc}
L:&\DD(\RR)\times \BB_0 &\rightarrow & \R\\
  & (f, u) & \mapsto    & D \widehat{G}_1 \left(f\right) u,
\end{array} 
\end{equation}
and
\begin{equation}
\label{equation definition K}
\begin{array}{rccc}
K:& \Omega \times \DD(\TT) \times \BB &\rightarrow & \R\\
  &(\omega, f, v) & \mapsto    &\displaystyle m\left(
D G_1 \left(\omega, f \right) v \right),
\end{array} 
\end{equation}
where $m$, $G_1$ and $\widehat{G}_1$ are the functions given by 
(\ref{equation definition of minim}), 
(\ref{equation definition G1}) and
(\ref{equation definition widehatG1}). 

Note that the map $L$ is linear on the component $u$
and non-linear but smooth with respect to the component $f$.
On the other hand $K$ is not linear on the vectorial component
$v$, but for any constant $k>0$ we have that $k K(\omega,f,v) = K(\omega,f,kv)$.
If $D G_1 \left(\omega, f , v\right)$ has a unique 
minimum as a function from $\T$ to $\R$
then we have that $K$ is differentiable in a neighborhood $V\subset  \DD(\TT) \times \BB$ of
$(f,v)$ (see appendix A in \cite{JRT11a}). 


Note that theorem \ref{thm symmetry reduction} is applicable to 
the family $\{c(\alpha,\eps)\}_{(\alpha,\eps)\in A}$. We can 
replace the value of $\alpha_n'(\omega_0)$ given 
by (\ref{equation alpha n + section}) and, rearranging the terms, 
we obtain 
\[
\frac{\alpha_n'(\omega_0)}{\alpha_{n-1}'(\omega_0)} = 
\frac{K\left(\omega_{n-1},f^{(n)}_{n-1}, \tv_{n-1}^{(n)} \right)  } 
{\rule{0ex}{3.5ex} K\left(\omega_{n-2},f^{(n-1)}_{n-2}, \tv_{n-2}^{(n-1)} \right)  } 
\cdot
\frac{  L\left(f^{(n-1)}_{n-2},  u_{n-2}^{(n-1)}\right)  } 
{\rule{0ex}{3.5ex}  L\left(f^{(n)}_{n-1},  u_{n-1}^{(n)}\right)  } .
\]

Consider \[
A_n(c)= \frac{  L\left(f^{(n-1)}_{n-2} ,  u_{n-2}^{(n-1)} \right)  } 
{\rule{0ex}{3.5ex}  L\left(f^{(n)}_{n-1} ,   u_{n-1}^{(n)}\right)  }  
\text{ and } 
B_n(c)= \frac{K\left(\omega_{n-1},f^{(n)}_{n-1} , \tv_{n-1}^{(n)} \right)  } 
{\rule{0ex}{3.5ex} K\left(\omega_{n-2},f^{(n-1)}_{n-2}, \tv_{n-2}^{(n-1)} \right)  }. 
\]

Using lemma \ref{lemma equivalencia multiplicacio} it is enough to 
prove that $\{A_n(c)\}_{n\geq0}, \{B_n(c)\}_{n\geq0} \in \ell^{\infty}(\R)$ and 
\[
A_n(c) \sim \deltabf^{-1} , \quad B_n(c) \sim   \frac{
\displaystyle K\left( \omega_{n-1}, f^*_{1}, \frac{v_{n-1}}{\|v_{n-1}\|}\right) 
}{ \rule{0pt}{4ex} \displaystyle 
K\left( \omega_{n-2}, f^*_{1}, \frac{v_{n-2}}{\|v_{n-2}\|}\right)
}\cdot 
\left\| D\TT_{\omega_{n-2}}(\Phi) \frac{ v_{n-2}}{\|v_{n-2}\|}.
\right\|. 
\]

Recall that $f^{(k)}_{k-1}$ corresponds to 
$\RR^{n} \left(\{c(\alpha,0)\} \right) \cap \Sigma_1$.
Using that the family $\{c(\alpha,0)\}$ crosses transversaly the stable 
manifold of the fixed point $\Phi$ of $\RR$ and that the set $\Sigma_1$ 
crosses transversely the unstable one dimensional manifold 
of $\Phi$, we have that $f^{(k)}_{k-1}$ converges 
geometrically to $\{f^*_{1}\} = W^u(\RR, \Phi) \cap \Sigma_1$, where 
$W^u(\RR,\Phi)$ is the unstable manifold of $\RR$ at the 
fixed point $\Phi$. Concretely we have that $f^{(k)}_{k-1} \sim f^*_{1}$. 

Recall that $L(f,u)= D\hat{G}_1(f) u$, therefore we have that 
$D_f L(f,u) v= D^2 \hat{G}_1(f) (u,v)$. 
We can apply now lemma \ref{lemma equivalencia funcio diferenciable} to 
$L\left(f^{(n)}_{n-1}, \frac{u_{n-1}^{(n)}}{\|  u_{n-1}^{(n)} \|} \right)$, 
then we have  that 
\begin{equation}
\label{equation demo assymptotic alphas 1}
L\left(f^{(n)}_{n-1} , \frac{u_{n-1}^{(n)}}{\|  u_{n-1}^{(n)} \|} \right)  
\sim L\left(f^{*}_{1} ,  \frac{u_{n-1}^{(n)}}{\|  u_{n-1}^{(n)} \|} \right).
\end{equation}

On the other hand, we have that $u^{k}_0$ converges geometrically to 
$u_0 = \partial_\alpha c(\alpha^*,0)$ with $\alpha^*$ the 
parameter value for which the family $\{c(\alpha,9)\}_{\alpha,0} \in A$
intersects with $W^s(\Phi,\RR)$. Then, using the $\lambda$-lemma, we have 
that $\frac{u^{(k)}_{k-1}}{u^{(k)}_{k-1}}$ converges to $e^u(f^*_1)$, 
the unitary tangent vector to $W^u(\Phi,\RR)$ at the point $f^*_1$. 
With the use of the $\lambda$-lemma we also have that $\|u_{k-1}^{(k)}\|$ behaves asymptotically 
as $\deltabf^k$ when $k$ goes to infinity, with $\deltabf$ the unstable 
eigenvalue of $D\RR(\Phi)$. Concretely we have that $\|u_{n-1}\| > C_0$ for 
any $n\geq n_0$. 
Then we can multiply and divide 
$A_n(c)$ by $\| u_{n-2}^{(n-1)} \|$ and 
$\|  u_{n-1}^{(n)} \|$ and write 
\begin{equation}
\label{equation demo assymptotic alphas 0}
A_n(c)= \frac{  L\left(f^{(n-1)}_{n-2} ,
\frac{ u_{n-2}^{(n-1)}}{\|   u_{n-2}^{(n-1)} \|} \right) 
} 
{\rule{0ex}{3.5ex}  L\left(f^{(n)}_{n-1},
\frac{u_{n-1}^{(n)}}{\|  u_{n-1}^{(n)} \|} \right) 
} \cdot
\frac{ \|   u_{n-2}^{(n-1)} \|}{\|  u_{n-1}^{(n)} \|} .
\end{equation}

Also recall that $L$ is linear in the second component and 
 $\frac{u_{k-1}^{(k)}} {\|u_{k-1}^{(k)\|}}\sim e^u(f^*_1)$. Then it follows that 
\begin{equation}
\label{equation demo assymptotic alphas 2}
 L\left(f^{*}_{1} ,  \frac{u_{n-1}^{(n)}}{\|  u_{n-1}^{(n)} \|} \right)
\sim  L\left(f^{*}_{1} , e^u(f^*_1) \right).
\end{equation}

Note that the term $L\left(f^{*}_{1} , e^u(f^*_1) \right)$ is constant, and it is 
different from zero since $\Sigma_1$ crosses transversely $W^u(\Phi,\RR)$. 
Then we have that 
\[\frac{  L\left(f^{(n-1)}_{n-2} ,
\frac{ u_{n-2}^{(n-1)}}{\|   u_{n-2}^{(n-1)} \|} \right) 
} {\rule{0ex}{3.5ex}  L\left(f^{(n)}_{n-1},
\frac{u_{n-1}^{(n)}}{\|  u_{n-1}^{(n)} \|} \right) 
} \sim 1.\]
We also have that $\|u_{k-1}^{(k)}\|\sim \deltabf^k$, therefore 
$\frac{ \|   u_{n-2}^{(n-1)} \|}{\|  u_{n-1}^{(n)} \|} \sim \deltabf^{-1}$. 
Applying this to (\ref{equation demo assymptotic alphas 0}) we have 
$A_n(c) \sim \deltabf^{-1}$. Additionally, this implies that
$A_n(c)\in \ell^{\infty}(\R)$.

Now we  focus on the asymptotics of $B_n(c)$. We follow  
the same arguments used for the study of $A_n(c)$. 
Using that $C K(\omega,f,v) = 
K(\omega,f,C v)$ for any constant $C>0$ and the fact that 
(due to conjecture {\bf \ref{conjecture H3}}) there exists $C>0$ 
such that $\| v^{(n)}_{n-1} \|> C$ 
for any $n$, we can  rearrange the terms on the expression of $B_n(c)$ in 
such a way that we have
\begin{equation}
\label{equation demo assymptotic alphas 5}
B_n(c)= \frac
{ K\left(\omega_{n-1},f^{(n)}_{n-1},
\frac{\tv_{n-1}^{(n)}}{\|  \tv_{n-1}^{(n)} \|} \right) 
} {\rule{0ex}{3.5ex}  K\left(\omega_{n-2}, f^{(n-1)}_{n-2} , 
\frac{ \tv_{n-2}^{(n-1)}}{\|  \tv_{n-2}^{(n-1)} \|} \right) } 
\cdot
\frac {\|  \tv_{n-1}^{(n)} \|} { \|   \tv_{n-2}^{(n-1)} \|} .
\end{equation}

Consider the term $\frac{ \|   \tv_{n-2}^{(n-1)} \|}{\|  \tv_{n-1}^{(n)} \|}$, note 
that we can use the same argument as we used for 
$\frac{ \|   u_{n-2}^{(n-1)} \|}{\|  u_{n-1}^{(n)} \|}$ to conclude that 
\[
\frac  {\|  \tv_{n-1}^{(n)} \|}{ \|   \tv_{n-2}^{(n-1)} \|} \sim  
\frac  {\|  v_{n-1} \|}{ \|   v_{n-2}  \|} = 
\left\|t_{\gamma(v_{n-2})} \left( D\TT_{\omega_{n-2}}(f_2^*) 
\frac{ v_{n-2}}{\|v_{n-2}\|}\right) \right\| = \left\| D\TT_{\omega_{n-2}}(f_2^*) 
\frac{ v_{n-2}}{\|v_{n-2}\|}\right\|. 
\]

On the other hand, using  the hypothesis {\bf H3} and applying
lemma \ref{lemma equivalencia funcio diferenciable}  
to  $ K\left(\omega_{n-1},f^{(n)}_{n-1},
\frac{\tv_{n-1}^{(n)}}{\|  \tv_{n-1}^{(n)} \|} \right)$, we have 
\[
K\left(\omega_{n-1},f^{(n)}_{n-1},
\frac{\tv_{n-1}^{(n)}}{\|  \tv_{n-1}^{(n)} \|} \right) 
\sim
K\left(\omega_{n-1},f^*_1,
\frac{\tv_{n-1}^{(n)}}{\|  \tv_{n-1}^{(n)} \|} \right).
\]
Using the hypothesis {\bf H3} again we have that 
$K\left(\omega_{n-1}, f^*_1, \frac{\tv_{n-1}^{(n)}}{\|  \tv_{n-1}^{(n)} \|} \right)$ 
is differentiable with respect to the third component. Then using the mean 
value theorem and conjecture {\bf \ref{conjecture H3}} it can be shown
(by means of an analog argument to the one used in the proof of 
lemma \ref{lemma equivalencia funcio diferenciable})  that 
\[
K\left(\omega_{n-1},f^*_1,
\frac{\tv_{n-1}^{(n)}}{\|  \tv_{n-1}^{(n)} \|} \right) 
\sim 
K\left(\omega_{n-1},f^*_1,
\frac{v_{n-1}}{\|  v_{n-1} \|} \right).
\]

It is only left to check that $B_n(c) \in \ell^{\infty}(\R)$. Using the last 
part of conjecture {\bf \ref{conjecture H3}} we have that the sequence given as  
$\left\{1/K\left(\omega_{n-2},  f^*_1,  \frac{v_{n-2}}{\|v_{n-2}\|} 
\right)\right\}_{n\geq0}$ is bounded.  On 
the other hand, using the definition of the operator $K$ given by 
(\ref{equation definition K}) and the proposition 
3.21 in \cite{JRT11a} it follows that  $\left\{K\left(\omega_{n-1},  
f^*_1,  \frac{v_{n-1}}{\|v_{n-1}\|} 
\right)\right\}_{n\geq0}$ is also a bounded sequence. 
Note that 
$D\TT_{\omega}(\Phi)$ is a bounded operator for any $\omega\in\T$,
therefore we 
have that $\left\{\left\| D\TT_{\omega_{n-2}}(f_2^*) \frac{ v_{n-2}}{\|v_{n-2}\|}
\right\|\right\}_{n\geq 0} \in \ell^{\infty}(\R)$. Using lemma 
\ref{lemma equivalencia multiplicacio}
it follows that $B_n(c) \in \ell^{\infty}(\R)$, which finishes the proof. 
\end{proof}

\subsection{Theoretical explanation to the first numerical observation} 
\label{subsection explanation 1st observation}
Consider a two parametric family of maps $\{c(\alpha,\eps)\}_{(\alpha,\eps)
\in A}$ contained in $\BB$ satisfying the hypotheses {\bf H1}, {\bf H2}
and {\bf H3}. Consider also $\omega_0$ a  Diophantine rotation 
number for the family. As in the previous section 
we are concerned with the asymptotic behavior 
of the reducibility loss directions $\alpha'_n(\omega_0,c)$.

Due to theorem \ref{reduction alphas 
to dynamics of renormalization}
 we have that the values 
$\frac{\alpha'_n(\omega_0,c)}{\alpha'_{n-1}(\omega_0,c)}$ 
depend only on the sequences $\omega_n$ and $v_n$ given by 
equation (\ref{equation sequences directions simplified}), 
with $v_0=t_{\gamma_0}\left(\partial_\eps c(\alpha^*,0) \right)$, 
$\gamma_0$ such that $v_0 \in \BB'$ and
$\alpha^*$ the parameter value for which the family intersects 
$W^s(\RR,\Phi)$. The behavior of vectors $v_n$ is described by 
the dynamics of the following operator,
\begin{equation}
\label{equation map linearized q.p. renormalization general} 
\begin{array}{rccc}
L:& \T \times \BB' &\rightarrow &  \T \times \BB'  \\
\rule{0pt}{5ex} & (\omega, v) & \mapsto & \left(2\omega,  
\displaystyle  \frac{t_{\gamma(v)} \left( D\TT_\omega(\Phi) v\right)}
{\|t_{\gamma(v)} \left(D\TT_\omega(\Phi) v\right) \|} \right),
\end{array}
\end{equation} 
where $\gamma$ is chosen such that $t_{\gamma(v)} 
\left( D\TT_\omega(\Phi) v\right)$ belongs to $\BB'$.

In this section we focus in the case where $\{c(\alpha,\eps)
\}_{(\alpha,\eps)\in A}$ satisfies also hypothesis {\bf H2'}. 
In such a case, we have that $v_0=\partial_\eps c(\alpha^*,0)$  
belongs to $\BB_1$  the linear subspace of $\BB$ given by
(\ref{equation spaces bbk}) for $k=1$.  Due to proposition 2.16 in \cite{JRT11a} 
the space $\BB_1$ is invariant by the iterates of $D\TT_\omega(\Psi)$. 

Consider $\LL_\omega$ is the map defined by equation (\ref{equation maps L_omega})  
(this is the restriction of  $D\TT_\omega(\Psi)$ to $\BB_1$). 
Let us define 
\begin{equation}
\label{equation map LL_omega'}
\begin{array}{rccc}
\LL_\omega': &   \BB_1'  & \rightarrow &  \BB_1'   \\
&\rule{0ex}{3ex} v  & \mapsto & \displaystyle t_{\gamma(v)} \circ \LL_\omega(v) , 
\end{array}
\end{equation}
where $\gamma(v)$ is chosen such that 
$t_{\gamma(v)} \circ \LL_\omega(v)
\in \BB_1'$. Note that, due to proposition 
\ref{proposion projections of BB1p} above, the value $\gamma$ is 
unique.
Actually, we can use this map to induce the following one on 
$\T \times \BB_1' $,
\begin{equation}
\label{equation map L_omega'}
\begin{array}{rccc}
L_1: &  \T\times \BB_1' & \rightarrow & \T\times \BB_1'  \\ 
     & \rule{0ex}{4.5ex}(\omega, v)  & \mapsto & \displaystyle 
\left( 2\omega, \frac{\LL_\omega'(v) } 
{\| \LL_\omega'(v)\| }\right).
\end{array}
\end{equation}
This is the restriction of $L_1$ to $\T\times \BB_1'$.

In \cite{JRT11c} we present numerical evidences which suggest that 
the following conjecture is true.  

\begin{conj}
\label{conjecture H4}
There exists an open set $V\subset 
\BB_1'$ (independent of $\omega$)  such that 
the second component of the map $L_1'$ given by 
(\ref{equation map L_omega'}) is contractive (with the supremum norm) 
in the unit sphere and it maps the set $V$ into itself 
for any $\omega \in \T$. Additionally we will assume that the contraction 
is uniform for any $\omega \in \T$, in the sense that there exists 
a constant $0<\rho<1$ such that the Lipschitz constant 
associated to the second component 
of the map $L_1'$ is upper bounded by $\rho$ for any $\omega\in \T$. 
\end{conj}

Consider 
\begin{equation}
\label{Set Rot(v)} 
\operatorname{Rot}(V) = \left\{ v \in \BB_1 \thinspace
 | \thinspace t_\gamma(v) \in V \subset \BB_1' \text{ for some } \gamma \in \T\right\} .
\end{equation}

The following result gives a theoretical explanation to 
the first numerical observation described in the 
introduction (section \ref{sec:num obs}). 

\begin{thm}
\label{theorem proof of the universality conjecture} 
Consider $\{c_1(\alpha,\eps)\}$ and $\{c_2(\beta,\eps)\}$ 
two different families of two parametric 
maps satisfying the hypotheses {\bf H1}, {\bf H2'} and {\bf H3}. 
Assume that conjectures {\bf \ref{conjecture H2}}, {\bf \ref{conjecture H3}} and 
{\bf \ref{conjecture H4}} are true. Let $\alpha^*$ and $\beta^*$ be the parameter 
values where each family $c_1(\alpha,0)$ and $c_2(\beta,0)$ 
intersects $W^s(\RR,\Phi)$, the 
stable manifold of the fixed point of the renormalization operator.  
Assume that $\partial_\eps c_1(\alpha^*,0)$ and 
$ \partial_\eps c_2(\beta^*,0)$ belong to $\operatorname{Rot}(V)$. 

Then, for any $\omega_0\in \Omega$,  we have that 
\begin{equation}
\label{equation thm proof universality conjecture} 
\frac{\alpha_n'(\omega_0,c_1)}{ \alpha_{n-1}'(\omega_0,c_1)}
\sim
\frac{\alpha_n'(\omega_0,c_2)}{ \alpha_{n-1}'(\omega_0,c_2)},
\end{equation}
where $\alpha_i'(\omega_0, c_i)$  are the reducibility loss 
directions associated to each family $c_i$ for
the rotation number of the system equal to $\omega_0$. 
\end{thm}


\subsubsection{Proofs}

\begin{lem}
\label{lemma reduction alphas to dynamics of renormalization}
Let  $c_1$ and $c_2$ be two families of q.p. forced maps satisfying 
the hypotheses of theorem \ref{theorem proof of the universality 
conjecture} and $\omega_0$ a
Diophantine number.  Consider the loss of reducibility
directions $\alpha_n'(c_i,\omega_0)$ associated to each family of
map, as well as  $v_n(c_i,\omega_0)$ given by the
recurrence (\ref{equation sequences directions simplified}).
 
If
\begin{equation}
\label{equation hyphotesis lemma reduction alphas}
 \frac{ v_{k}(\omega_0,c_1)}{\|v_{k}(\omega_0,c_1)\|} 
\sim \frac{ v_{k}(\omega_0,c_2)}{\|v_{k}(\omega_0,c_2)\|}, 
\end{equation}
then we have that
\begin{equation}
\label{equation lemma reduction alphas}
\frac{\alpha_n'(\omega_0,c_1)}{ \alpha_{n-1}'(\omega_0,c_1)}
\sim
\frac{\alpha_n'(\omega_0,c_2)}{ \alpha_{n-1}'(\omega_0,c_2)}.
\end{equation}
\end{lem}
\begin{proof}
Using the same arguments of theorem \ref{reduction alphas to dynamics 
of renormalization} it follows that the sequences
$$
\left\{ K\left(\omega_{n-1}, f^*_{1} , v_{n-1}'(\omega_0,c_i) \right) 
\right\}_{n>0},$$ 
$$\left\{ 1/ K\left(\omega_{n-2},  f^*_1,  v_{n-2}'(\omega_0,c_i) \right) 
\right\}_{n>1},$$  and 
$$\left\{ \left\| D\TT_{\omega_{n-2}}(f_2^*) v_{n-2}'(\omega_0,c_i) 
\right\| \right\}_{n>1},$$
belong to $\ell^\infty(\RR)$ for $i=1,2$. 

On the other hand, using the condition (\ref{equation 
hyphotesis lemma reduction alphas}) given by hypothesis and
using the differentiability of $K$ is not difficult to see that
\[
K\left(\omega_{k}, f^*_{1} , \frac{v_{k}(\omega_0,c_1) }{\|v_{k}(\omega_0, c_1) \|} 
\right) \sim 
K\left(\omega_{k}, f^*_{1} , \frac{v_{k}(\omega_0,c_2) }{\|v_{k}(\omega_0, c_2) \|} \right), 
\]
for $k=n-1, n-2$. Then  using that $D\TT_{\omega_{n-2}}(f_2^*)$ is linear we 
have that
\[
\left\| D\TT_{\omega_{n-2}}(f_2^*) \frac{ v_{n-2}(\omega_0, c_1)}
{\|v_{n-2}(\omega_0, c_1)\|} \right\| \sim
\left\| D\TT_{\omega_{n-2}}(f_2^*) \frac{ v_{n-2}(\omega_0, c_2)}
{\|v_{n-2}(\omega_0, c_2)\|} \right\|.
\]

Finally, we can apply
(\ref{equation theorem reduction to dynamics of renormalization})
and lemma \ref{lemma equivalencia multiplicacio} to conclude 
that (\ref{equation lemma reduction alphas}) holds.
\end{proof}

\begin{proof}[Proof of theorem \ref{theorem proof of the universality conjecture}]

Using lemma \ref{lemma reduction alphas to dynamics of renormalization}
it is enough to prove that (\ref{equation hyphotesis lemma reduction alphas}) 
holds. 

Given $\Psi$ a point where $\TT_\omega$ is differentiable, 
consider the map
\[
\begin{array}{rccc}
H_{\Psi,\omega}: &   \BB_1'  & \rightarrow &  \BB_1'   \\
&\rule{0ex}{3ex} v  & \mapsto & \displaystyle 
t_{\gamma(v)} \left(D\TT_\omega(\Psi)v\right) , 
\end{array}
\]
where $\gamma(v)$ is chosen such that
$t_{\gamma(v)} \left(D\TT_\omega(\Psi)v\right) 
\in \BB_1'$. Let $\Phi$ be the fixed point of the 
quasi-periodic renormalization operator, then 
$H_{\Phi,\omega)} \equiv \LL_\omega'$.

Using conjecture {\bf \ref{conjecture H4}} and the 
differentiability of $\TT_\omega$ we have that there 
exists a neighborhood of $\Phi$ such that, for any 
$v_1,v_2\in V\subset B_1'$ and $\omega\in \T$, we have
\begin{equation}
\label{equation bound contracivenes 2} 
\left\|\frac{H_{\Psi,\omega} (v_1)}{\| H_{\Psi,\omega} (v_1)\|} - 
\frac{H_{\Psi,\omega} (v_2)}{\| H_{\Psi,\omega} (v_2)\|} 
 \right\|<\tilde{\rho} \left\|\frac{v_1}{\|v_1\|}- 
\frac{v_2}{\|v_2\|}\right\|,
\end{equation} 
with $0<\tilde{\rho} <1$. Note that the invariance of 
$V$ given by conjecture {\bf \ref{conjecture H4}} also extends 
to this neighborhood of $\Phi$. 

Consider $f^*_j= \Sigma_j\cap W^u(\RR,\Phi)$, since 
$\Sigma_j$ accumulate to $W^s(\RR,\Phi)$,  we have that there 
exists $j_0$ such that $f^*_j$ belong to a neighborhood of $\Phi$ 
arbitrarily small. Using this fact, 
the definition of $v_{k}(\omega_0,c_i)$ given by 
(\ref{equation sequences directions simplified}), 
 (\ref{equation bound contracivenes 2})
and the differentiability of $\TT_\omega$ it is not difficult to 
check that 
\[
\left\| \frac{ v_{n}(\omega_0,c_1)}{\|v_{n}(\omega_0,c_1)\|} 
- \frac{ v_{n}(\omega_0,c_2)}{\|v_{n}(\omega_0,c_2)\|} \right\| 
\leq K_0 \tilde{\rho}^{n-j_0} \left\|\frac{ v_{0}(\omega_0,c_1)}{\|v_{0}(\omega_0,c_1)\|} 
- \frac{ v_{0}(\omega_0,c_2)}{\|v_{0}(\omega_0,c_2)}\right\|  
\]
where $K_0$ a constant.  This implies that 
(\ref{equation hyphotesis lemma reduction alphas}) holds, which finishes the proof.
\end{proof}

\begin{rem}
\label{remark on the assumption H4)}
To prove theorem 
 \ref{theorem proof of the universality conjecture} conjecture 
{\bf \ref{conjecture H4}} can be slightly relaxed. On the one hand, 
the existence of the open set $V\subset B'_1$ such that 
$L_1(\omega,V)\subset \{2\omega\}\times V$ can be replaced for 
an open set $V(\omega)$ depending on $\omega$ such that 
$L_1(\omega,V)\subset \{2\omega\}\times V(2\omega)$ for each 
$\omega\in\T$.  On the other hand, the contractivity on the 
second component can be replaced by the following condition. 
There exists constants $0<\rho<1$, $K>0$  and $n_0\in \N$
(independent of $\omega$), such that 
\[
\|\pi_2\left[(L_1')^n(\omega,u)\right] 
- \pi_2\left[(L_1')^n(\omega,v)\right]\| < K \rho^n
\]
for any $(\omega, u), (\omega,v) \in  \{\omega\}\times V(\omega)$  
(with $u$ and $v$ unitary vectors)  and any $n>n_0$.
\end{rem}

\subsection{Theoretical explanation to the second numerical observation} 
\label{subsection explanation 2nd observation}

In this section we give a theoretical explanation of the second 
numerical observation described in section \ref{sec:num obs}. 

Given a two parametric family of maps $\{c(\alpha,\eps)\}_{(\alpha,\eps)
\in A}$ satisfying hypotheses {\bf H1} and {\bf H2} let 
$\alpha'_n(\omega, c)$ denote the slope of the reducibility loss bifurcation 
associated to the $2^n$ periodic invariant curve of the family. 
Let $\left\{\TT_\omega\left(c(\alpha,\eps)\right)\right\}_{(\alpha,\eps)
\in \tilde{A}}$ denote the family defined as the renormalization 
$\{c(\alpha,\eps)\}_{(\alpha,\eps) \in A}$. 
This is, let $f=f_{\alpha,\eps}: \B_\rho \times \W \rightarrow 
\W$ be the map which defines $c(\alpha,\eps)$, then 
$\TT_\omega(c(\alpha,\eps))$ 
is given by the map $g: \B_\rho \times \W \rightarrow 
\W$ defined as $g =\TT_\omega(f)$. 
Let $\alpha^*$ be the parameter value for which $\{c(\alpha,0)\}_{(\alpha,0)
\in A}$ intersects $W^s(\Phi,\RR)$. For $(\alpha, \eps)$ close 
enough to $(\alpha^*,0)$ we have that $\TT_\omega(f_{\alpha,\eps})$ is 
well defined, then $\TT_\omega(c)$ is also well defined family for 
$(\alpha,\eps)\in \tilde{A}\subset A$ a neighborhood of $(\alpha^*,0)\in A$. 
If the family $\{c(\alpha,\eps)\}_{(\alpha,\eps) \in A}$ has as 
associated rotation number $\omega$, then the family 
$\left\{\TT_\omega\left(c(\alpha,\eps)\right)\right\}_{(\alpha,\eps)
\in \tilde{A}}$ has as a rotation number $2\omega$.

\begin{thm}
\label{theorem universality implies renormalization}
Assume that there exists $B_0$ a set of two parametric families (satisfying 
the hypotheses {\bf H1} and {\bf H2}) such that: 
\begin{enumerate}
\item \label{codition 1 thm univ implies renor} 
For any $c_1$ and $c_2$ in $B_0$, we have that
\[
 \frac{\alpha'_i(\omega, c_1)}{\alpha'_{i-1}(\omega,c_1)} 
\sim
\frac{\alpha'_i(\omega, c_2)}{\alpha'_{i-1}(\omega,c_2)}.
\]

\item \label{codition 2 thm univ implies renor}
For any family $\{c(\alpha,\eps)\}_{(\alpha,\eps)
\in A} \in B_0$ we have that $\left\{\TT_\omega\left(
c(\alpha,\eps)\right)\right\}_{(\alpha,\eps) \in \tilde{A}} \in B_0$.

\item \label{codition 3 thm univ implies renor}
For any value $\omega$ we have that
$\alpha'_i(\omega,c)/\alpha'_i(2\omega,c)$ is a bounded sequence.
\end{enumerate}
Then $\alpha'_i(\omega,c) /\alpha'_{i-1}(2\omega,c)$
converges to a constant value.
\end{thm}

We want to use this theorem  and theorem \ref{theorem proof of the 
universality conjecture} to give an explanation of the second of the 
numerical observations described in section \ref{sec:num obs}. To 
do that we need to introduce a new conjecture. This conjecture will 
be also used in section \ref{subsection explanation 3rd observation} 
to explain the third numerical observation.

\begin{conj}
Consider  $\LL_\omega$ 
the map given by (\ref{equation maps L_omega}) and $\omega_0 \in \Omega$. 
Given  $v_{0,1}$ and $v_{0,2}$ two vectors in $\RHH(\W_{\rho})
\oplus \RHH(\W_{\rho})\setminus\{0\}$, consider the sequences 

\begin{equation}
\label{equation sequences directions counterexample 0}
\begin{array}{rcll} 
\omega_k & = & 2 \omega_{k-1} 
&
\text{ for } k = 1, ..., n-1.\\
\rule{0ex}{4ex} 
v_{k,1} & = & \LL_{\omega_{k-1}} \left(v_{k-1,1} \right)
&  \text{ for } k = 1, ..., n-1.\\
\rule{0ex}{4ex} 
v_{k,2} & = & \LL_{2\omega_{k-1}}\left(v_{k-1,2} \right)
&  \text{ for } k = 1, ..., n-1.
\end{array}
\end{equation}

\label{conjecture H5}
There exist constant $C_1$ and $C_2$ such that
\[
C_1 \frac{\|v_{0,2} \|}{\|v_{0,1}\|} \leq \frac{\|v_{n,2} \|} {\|v_{n,1}\|}  
\leq C_2 \frac{\|v_{0,2} \|}{\|v_{0,1}\|}, 
\]
for any $n\geq0$.
\end{conj}

In \cite{JRT11c} we include numerical evidences which suggest that this 
conjecture is true. It can be interpreted as a uniform growth condition 
on $\LL_\omega$.

\begin{cor}
\label{corolary univ implies renor}
Consider $\{c(\alpha,\eps)\}_{(\alpha,\eps)\in A}$ a two parametric
maps satisfying the hypotheses {\bf H1}, {\bf H2'} and {\bf H3}.
Assume that conjectures {\bf \ref{conjecture H2}}, {\bf \ref{conjecture H3}}, 
{\bf \ref{conjecture H4}} and {\bf \ref{conjecture H5}} are true. Let $\alpha^*$ 
be the parameter values for which the family $\{c(\alpha,0)\}_{(\alpha,0)\in A}$ 
intersects $W^s(\RR,\Phi)$. Consider $\operatorname{Rot}(V)$ the set 
given by (\ref{Set Rot(v)}) and assume that $\partial_\eps c(\alpha^*,0) \in \operatorname{Rot}(V)$.

Then for any $\omega_0\in \Omega$  we have that
\begin{equation}
\label{equation coro univ implies renor} 
\lim_{n\rightarrow \infty} 
\frac{\alpha_n'(\omega_0,c_1)}{ \alpha_{n-1}'(2\omega_0,c_1)},
\end{equation}
exists.
\end{cor}

\subsubsection{Proofs} 

\begin{proof}[Proof of theorem \ref{theorem universality implies renormalization}]

To prove the result we will need the following lemmas

\begin{lem}
\label{lemma 1 proof thm univ implies renor} 
Given a sequence $\{s_i\}_{i\in \Z_+}$ such that $s_i \sim s_{i-1}$
(in other words that its equivalent to the same sequence shifted by one position)
then it converges to a limit.
\end{lem}
\begin{proof}
We will see that $\{s_i\}_{i\in \Z_+}$ is a Cauchy sequence.

Consider a positive integer $N_0$ fixed, then for any (positive
integer) $N$ we have that
\[
s_{N_0} - s_{N_0 + N} = \sum_{i= N_0}^{N_0 + N -1} s_{i} - s_{i+1}.
\]
Using this and the triangular inequality we have that
\[
| s_{N_0} - s_{N_0 + N} | \leq \sum_{i= N_0}^{N_0 + N -1} | s_{i} - s_{i+1}|.  
\]
Now we can use the hypothesis $s_i \sim s_{i+1}$ and
bound term by term the equation above to obtain
\[
| s_{N_0} - s_{N_0 + N} | \leq \sum_{i= N_0}^{N_0 + N -1}  K_0 \rho^i 
\leq K_0 \frac{\rho^{N_0}}{1- \rho}. 
\]
Therefore $\{s_i\}_{i\in \Z_+}$ is a Cauchy sequence.
\end{proof}

\begin{lem}
\label{lemma 2 proof thm univ implies renor}
Consider a two parametric family of maps $\{c(\alpha,\eps)\}_{(\alpha,\eps)
\in A}$ satisfying hypotheses {\bf H1} and {\bf H2}. Let 
$\alpha'_n(\omega, c)$ denote the slope of the reducibility loss bifurcation 
associated to the $2^n$ periodic invariant curve of the family. Then we have 
$\alpha'_i(\omega,c) = \alpha'_{i-1}(2\omega,\TT_\omega(c))$.
\end{lem}
\begin{proof}
The proof relies on the concepts introduced in sections 3.1 and 
3.2 of \cite{JRT11a}, concretely let us consider the set 
$\Upsilon_i^{+}(\omega)$ introduced there. We have that $\alpha'_i
(\omega,c)$ is the slope at $\eps=0$ of the curve in $A$ defined by 
$\{c(\alpha,\eps)\}_{(\alpha,\eps)\in A} \cap \Upsilon^+_i(\omega)$. 
On the other hand,  $\alpha'_{i-1}(2\omega,\TT_\omega(c))$ 
is the tangent direction of the curve in $\tilde{A}\subset A$
defined by $\{\TT_\omega(c(\alpha,\eps))\}_{(\alpha,\eps)\in \tilde{A}} \cap 
\Upsilon^+_{i-1}(2\omega)$. 

Let $\alpha_i$ be the parameter value for which 
$\{c(\alpha,0)\}_{(\alpha,0)\in A}$ intersects 
$\Sigma_i = \Upsilon^+_i(\omega) \cap \BB_0$. 
Using lemma 3.7 in \cite{JRT11a} we have that 
$\TT_\omega\left(\Upsilon^+_i(\omega) \cap \DD(\TT)\right) = 
\Upsilon^+_{i-1}(2\omega) \cap \operatorname{Im}(\TT_\omega)$, 
then we have that $\{c(\alpha,\eps)\}_{(\alpha,\eps)\in A} 
\cap \Upsilon^+_i(\omega)$ and 
$\{\TT_\omega(c(\alpha,\eps))\}_{(\alpha,\eps)\in \tilde{A}} \cap 
\Upsilon^+_{i-1}(2\omega)$ are exactly the same set around 
$\alpha_i$. Then their slope at $\eps=0$ also coincide. 
\end{proof} 

\begin{lem}
\label{lemma 3 proof thm univ implies renor}
Consider $\{r_i\}_{i\in \Z_+}$ and $\{s_i\}_{i\in \Z_+}$ 
two sequences of real numbers, with $r_i \sim s_i$. Consider also 
third sequence $\{K_i\}_{i \in Z_+}$ which is bounded.
Then we have $K_i r_i \sim K_i s_i$.
\end{lem}
\begin{proof} 
The proof follows easily from the definition of $\sim$. 
\end{proof}

Now we can focus on the proof of theorem 
\ref{theorem universality implies renormalization}. Using lemma 
\ref{lemma 1 proof thm univ implies renor} 
is enough to prove that
$\frac{\alpha'_i(\omega,c)}{\alpha'_{i-1}(2\omega,c)} \sim 
\frac{\alpha'_{i-1}(\omega,c)}{\alpha'_{i-2}(2\omega,c)}$.
Due to lemma \ref{lemma 2 proof thm univ implies renor}
we have  that $\alpha'_i(\omega,c) = \alpha'_{i-1}\left(2\omega,
\TT_\omega(c)\right)$,
therefore we have
\[
\frac{\alpha'_i(\omega,c)}{\alpha'_{i-1}(2\omega,c)} = 
\frac{\alpha'_{i-1}\left(2\omega,\TT_\omega(c)\right)}{\alpha'_{i-1}(2\omega,c)}. 
\]
If we multiply and divide the fraction by
$\alpha'_{i-2}\left(2\omega,\TT_\omega(c)\right) = \alpha'_{i-1}(\omega,c)$ we have
\begin{equation} 
\label{equation demo universality implies renormalization}
\frac{\alpha'_i(\omega,c)}{\alpha'_{i-1}(2\omega,c)} = 
\frac{\alpha'_{i-1}\left(2\omega,\TT_\omega(c)\right)}
{\alpha'_{i-2}\left(2\omega,\TT_\omega(c)\right)}
\cdot
\frac{\alpha'_{i-1}(\omega,c)}{\alpha'_{i-1}(2\omega,c)}.
\end{equation}

Using the fist and the second hypotheses and theorem 
\ref{theorem proof of the universality conjecture} we have 
\[
\frac{\alpha'_{i-1}\left(2\omega,\TT_\omega(c)\right)}
{\alpha'_{i-2}\left(2\omega,\TT_\omega(c)\right)}
\sim
\frac{\alpha'_{i-1}(2\omega,c)}{\alpha'_{i-2}(2\omega,c)}.
\]

Consider now two general sequences $\{r_i\}_{i\in \Z_+}$ and $\{s_i\}_{i\in \Z_+}$,
with $r_i \sim s_i$, and a third sequence $\{K_i\}_{i \in Z_+}$ which is bounded.
Then is not hard to see that $K_i r_i \sim K_i s_i$.

By hypothesis we have that $\left\{\frac{\alpha'_{i-1}(\omega,c)}
{\alpha'_{i-1}(2\omega,c)}\right\}_{i\in\Z_+}$ is bounded, then we can
apply lemma \ref{lemma 3 proof thm univ implies renor} 
to (\ref{equation demo universality implies renormalization}) to
obtain
\[
\frac{\alpha'_i(\omega,c)}{\alpha'_{i-1}(2\omega,c)} 
\sim 
\frac{\alpha'_{i-1}(2\omega,c)}{\alpha'_{i-2}(2\omega,c)}
\cdot
\frac{\alpha'_{i-1}(\omega,c)}{\alpha'_{i-1}(2\omega,c)} = 
\frac{\alpha'_{i-1}(\omega,c)}{\alpha'_{i-2}(2\omega,c)}.
\]
\end{proof}

\begin{proof}[Proof of corollary \ref{corolary univ implies renor}]
Set $B_0$ the set of two parametric families such that 
satisfy {\bf H1} and {\bf H2'}. The result
follows applying theorem \ref{theorem universality implies renormalization}. 
Let us check that the hypotheses of the theorem are satisfied. 

Condition \ref{codition 1 thm univ implies renor} of theorem 
\ref{theorem universality implies renormalization} is satisfied 
thanks to theorem \ref{theorem proof of the universality conjecture}. 

If a family  $\{c(\alpha,\eps)\}_{(\alpha,\eps)
\in A}$ satisfies {\bf H1}, we have that  
$\{c(\alpha,0)\}_{(\alpha,0)\in A}$ has a 
full cascade of period doubling bifurcations
 (in the sense described in {\bf H1}). Then 
$\left\{\TT_\omega\left(c(\alpha,0)\right)
\right\}_{(\alpha,0) \in \tilde{A}} = 
\left\{\RR \left(c(\alpha,0)\right)
\right\}_{(\alpha,0) \in \tilde{A}}
$ also has a full cascade of period doubling bifurcations. 
Then $\left\{\TT_\omega\left(c(\alpha,\eps)\right)
\right\}_{(\alpha,\eps) \in \tilde{A}}$ also 
satisfies {\bf H1}. If a family  $\{c(\alpha,\eps)\}_{(\alpha,\eps)
\in A}$ satisfies {\bf H2'} then $\left\{\TT_\omega\left(c(\alpha,\eps)\right)
\right\}_{(\alpha,\eps) \in \tilde{A}}$ also does due to 
the invariance of $\BB_1$ by $D\TT_\omega (\Psi)$. 
We have that condition \ref{codition 2 thm univ implies renor} of theorem
\ref{theorem universality implies renormalization} is also satisfied. 
 
If we apply theorem 3.8 in \cite{JRT11a} to $\alpha'_i(\omega,c)$ and 
 $\alpha'_i(2\omega,c)$ we obtain 
\begin{equation}
\label{equation proof coro univ imply reonr 1} 
\frac{\alpha'_i(\omega,c)}{\alpha'_i(2\omega,c)} 
= 
 \frac{m \left(DG_1 \left(\omega_{n-1}, f^{(i)}_{i-1}\right) 
v_{i-1}^{(i)} (\omega,c) \right)}
{\rule{0ex}{3.5ex}  
m \left(DG_1 \left(\omega_{n-1}, f^{(i)}_{i-1}\right) 
v_{i-1}^{(i)} (2\omega,c) \right)} .
\end{equation}

Using the same arguments used in the proof of theorem 
\ref{reduction alphas to dynamics of renormalization} 
to (\ref{equation proof coro univ imply reonr 1})
one obtains:
\begin{equation}
\label{equation proof coro univ imply reonr 2} 
\frac{\displaystyle
\alpha'_i(\omega,c)}{\alpha'_i(2\omega,c)} 
\sim  
 \frac{\displaystyle m \left(DG_1 \left(\omega_{n-1}, f^*_1\right) 
\frac{v_{i-1} (\omega,c)}{\|v_{i-1}  (\omega,c) \|} \right)}
{\rule{0ex}{3.5ex} \displaystyle  
m \left(DG_1 \left(\omega_{n-1}, f^*_1\right) 
\frac{v_{i-1}(2\omega,c)}{\| v_{i-1} (2\omega,c)\|} \right)} \thinspace 
\frac{\|v_{i-1} (\omega,c) \|}{\| v_{i-1}  (2\omega,c)\|}.
\end{equation} 

Using conjecture {\bf \ref{conjecture H3}}, we have that 
$\displaystyle  m \left(DG_1 \left(\omega_{n-1}, f^*_1\right) 
\frac{v_{i-1}(2\omega,c)}{\| v_{i-1} (2\omega,c)\|} \right)$ 
is bounded away from zero. Then the boundedness of 
$\displaystyle\frac{
\alpha'_i(\omega,c)}{\alpha'_i(2\omega,c)} $ only 
depends on the boundedness 
$\displaystyle\frac{\|v_{i-1} (\omega,c) \|}{\| v_{i-1}  (2\omega,c)\|}$, 
which is given by conjecture {\bf \ref{conjecture H5}}. 

Then we have that condition \ref{codition 1 thm univ implies renor} of theorem
\ref{theorem universality implies renormalization} is also satisfied.
\end{proof}

\subsection{Theoretical explanation to the third numerical observation} 
\label{subsection explanation 3rd observation}

In sections \ref{subsection explanation 1st observation} and 
\ref{subsection explanation 2nd observation} we focussed the
discussion on the asymptotic behavior for 
families satisfying hypothesis
{\bf H2'}. The aim of this section is to illustrate what 
happens with maps that satisfy hypotheses {\bf H1},
{\bf H2} and {\bf H3},  but not {\bf H2'}. This is the 
main difference between the family of maps considered in the first and 
second numerical observations of section \ref{sec:num obs} 
and the family considered in the third one.

Let $\{c(\alpha,\eps)\}_{(\alpha,\eps) \in A}$  be a two parametric family of maps 
satisfying hypotheses {\bf H1}, {\bf H2} and {\bf H3}. Let
$\alpha'_n(\omega, c)$ denote the slope of the reducibility loss bifurcation
associated to the $2^n$ periodic invariant curve of the family. 
Finally consider $\omega_0$ a  Diophantine rotation
number for the family. Let $\alpha^*$ be the parameter value 
for which $\{c(\alpha,0)\}_{(\alpha,0)\in A}$ intersects $W^s(\Phi,\RR)$.
Additionally assume that 
\[\partial_\eps c(\alpha^*,0) = v_{0,1} + v_{0,2} \text{ with } 
v_{0,i} \in \BB_i, \quad i=1,2,\] where
the spaces $\BB_i$ are given by (\ref{equation spaces bbk}). 

In the third numerical observation of section \ref{sec:num obs} we have 
considered the family $c$ as above with 
\begin{equation}
\label{eq: initial vectors}
v_{0,1}= f_1(x) \cos(\theta), \quad v_{0,2} = \eta f_2(x) \cos(2\theta). 
\end{equation}
As the family depends on $\eta$, we denote by $c_\eta$ this 
concrete family. This parameter $\eta$ is considered in addition 
to the parameters $\alpha$ and $\eps$ of the family. In other 
words, for each $\eta\geq 0$, $c_\eta$ is a two parametric family of maps. 
Numerical computations in \cite{JRT11p2} suggest that the sequence 
$\alpha_n'(\omega_0,c_\eta)/ \alpha_{n-1}'(\omega_0,c_\eta)$ (for $\eta>0$)
is not asymptotically equivalent to $\alpha_n'(\omega_0,c_0)/ 
\alpha_{n-1}'(\omega_0,c_0)$,
but both sequences are $\eta$-close to each other. Here $c_0$ denotes 
the family $c_\eta$ for $\eta=0$.   We first discuss 
why they are not asymptotically equivalent.

Due to theorem \ref{reduction alphas 
to dynamics of renormalization}
 we have that the values
$\frac{\alpha'_n(\omega_0,c)}{\alpha'_{n-1}(\omega_0,c)}$
depend only on the sequences $\omega_n$ and $v_n$ given by
equation (\ref{equation sequences directions simplified}),
with $v_0=t_{\gamma_0}\left(\partial_\eps c(\alpha^*,0) \right)$,
$\gamma_0$ such that $v_0 \in \BB'$ and
$\alpha^*$ the parameter value for which the family intersects
$W^s(\RR,\Phi)$. 

Due to theorem \ref{reduction alphas to dynamics of renormalization}
we have that the values $\frac{\alpha'_n(\omega_0,c_\eta)}
{\alpha'_{n-1}(\omega_0,c_\eta)}$
depend only on the sequences $\omega_n$ and $v_n$ given by
(\ref{equation sequences directions simplified}),
with $v_0=t_{\gamma_0}\left(\partial_\eps c(\alpha^*,0) \right)$,
$\gamma_0$ such that $v_0 \in \BB'$ and
$\alpha^*$ the parameter value for which the family intersects
$W^s(\RR,\Phi)$.
Recall that the space $\BB_1$ and $\BB_2$ are invariant by 
$D\TT_\omega(\Phi)$ (see proposition 2.16 in \cite{JRT11a}). 
We have that $v_n$ can be written as 
\begin{equation}
\label{eq:spliting vn}
v_n = v_{n,1} + v_{n,2},
\end{equation}
with 
\begin{equation}
\label{equation sequences directions counterexample}
\begin{array}{rcll} 
\omega_k & = & 2 \omega_{k-1} 
&
\text{ for } k = 1, ..., n-1.\\
\rule{0ex}{4ex} 
v_{k,1} & = & t_{\gamma(v)}\left(\LL_{\omega_{k-1}} \left( v_{k-1,1} \right)\right) 
&  \text{ for } k = 1, ..., n-1.\\
\rule{0ex}{4ex} 
v_{k,2} & = & t_{\gamma(v)}\left(\LL_{2\omega_{k-1}}\left(  v_{k-1,2} \right)\right)
&  \text{ for } k = 1, ..., n-1.
\end{array}
\end{equation}
where $v_{0,1}$ and $v_{0,2}$ are given by (\ref{eq: initial vectors}) and 
the value $\gamma(v_{n-1})$ is chosen such that $t_{\gamma(v)}\left( 
D\TT_\omega(\Phi) v_{n-1}\right)$ belongs to $\BB'$. Note that, since 
$t_\gamma$ and the projection $\pi_1$ given by (\ref{projection pi_1}) commute, $\gamma(v_{n-1})$ 
only depends on $v_{n-1,1}$. Then we have that the vectors $\frac{v_{n-1,1}
(c_{\nu_1})}{\|v_{n-1,1}(c_{\nu_1})\|}$ and $\frac{v_{n-1,1}
(c_{\nu_2})}{\|v_{n-1,1}(c_{\nu_2})\|}$ are asymptotically equivalent
 for any $\nu_1\neq\nu_2$.  Despite of these, the vectors 
$\frac{v_{n-1,2}(c_{\nu_1})}{\|v_{n-1,2}(c_{\nu_1})\|}$ and 
$\frac{v_{n-1,2}(c_{\nu_2})}{\|v_{n-1,2}(c_{\nu_2})\|}$ will not 
be (in general) asymptotic equivalents.  This explains why the universal 
behavior of the sequence $\alpha_n'(\omega_0,c_\eta)/\alpha_{n-1}'
(\omega_0,c_\eta)$ ceases for $\eta>0$.  

\begin{rem}
If we have a family with $v_0\in \BB_{j}\oplus\BB_{k}$ (with $j\neq k$) 
instead of $v_0\in \BB_{1}\oplus\BB_{2}$, then the same discussion 
can be adapted with minor modifications. 
\end{rem}

\begin{rem}
Consider $\tilde{c}$ an arbitrary two parametric family 
satisfying the hypotheses {\bf H1}, {\bf H2} and {\bf H3}. If 
the Fourier expansion w.r.t $\theta$ of $\partial_\eps c(\alpha^*,0)$ 
has non-trivial Fourier nodes for 
different orders of the expansion, then one should not expect it 
to exhibit the universal behavior of the Forced Logistic Map, 
since the same argument used for the family $c_\eta$ would be applicable. 
\end{rem}

To explain why quotients  $\alpha_n'(\omega_0,c_\eta)/
\alpha_{n-1}'(\omega_0,c_\eta)$ and $\alpha_n'(\omega_0,c_0)/
\alpha_{n-1}'(\omega_0,c_0)$  are $\eta$-close we have the following 
result. 

\begin{thm}
\label{thm eta-close alphas}
Consider $\{c_\eta(\alpha,\eps)\}_{(\alpha,\eps) \in A}$ a family of maps 
satisfying the hypotheses {\bf H1}, {\bf H2} and {\bf H3} 
for any $\eta_0\geq \eta\geq 0$ (with $\eta_0\neq0$ fixed).
Assume that conjectures {\bf \ref{conjecture H2}}, {\bf \ref{conjecture H3}}
and
{\bf \ref{conjecture H5}} are true.

Then there exist $\tilde{\eta}_0$ sufficiently small such that,
 for any $\tilde{\eta}_0\geq \eta\geq 0$, we have that
\begin{equation}
\left| \frac{\alpha_n'(\omega_0,c_\eta)}{ \alpha_{n-1}'(\omega_0,c_\eta)}
- \frac{\alpha_n'(\omega_0,c_0)}{ \alpha_{n-1}'(\omega_0,c_0)}\right|
\leq O(\eta) + O(\rho^n),
\end{equation}
where $\rho$ is the constant $0<\rho<1$ associated to the 
asymptotic equivalence relation $\sim$.
\end{thm}

\begin{proof}[Proof of theorem \ref{thm eta-close alphas}] 
We need the following lemma. 
\begin{lem}
\label{prop ratios acotats implica distancies acotades}
Assume the same hypotheses as in theorem \ref{thm eta-close alphas}. 
Consider $v_k$, $v_{k,1}$ and $v_{k,2}$ given by 
(\ref{eq:spliting vn}) and 
(\ref{equation sequences directions counterexample}), with 
$v_{0,1}$ and $v_{0,2}$ given by (\ref{eq: initial vectors}).
Then we have that 
\[
\left\|\frac{v_k}{\|v_k \|}  - \frac{v_{k,1}}{\|v_{k,1} \|}\right\| <
\frac{ 2 C \eta}{1- C \eta} .
\]
\end{lem}
\begin{proof}
If we  use $v_n = v_{n,1} + v_{n,2}$, rearrange the sums,  and we 
apply the triangular inequality, then we have
\begin{eqnarray} 
\label{equation proof of prop close to universal}
\left\|\frac{v_n}{\|v_n \|}  - \frac{v_{n,1}}{\|v_{n,1} \|}\right\| & = & 
\left\|\frac{\|v_{n,1} \| - \|v_{n,1}  + v_{n,2}\|}{\|v_{n,1} \| \|v_{n,1}  + v_{n,2}\|}
v_{n,1}   +  \frac{v_{n,2}}{\|v_{n,1} +  v_{n,2} \|}\right\| \nonumber \\ 
& \leq & \frac{| \|v_{n,1} \| - \|v_{n,1}  + v_{n,2}\| | } {\|v_{n,1}  + v_{n,2}\|}
+  \frac{\|v_{n,2}\|}{\|v_{n,1} +  v_{n,2} \|}. 
\end{eqnarray}
Note that to deduce the last equation it is necessary to check that $\|v_{n,1}\| >0$. This 
is true due  to conjecture {\bf \ref{conjecture H3}}. Recall that if the assumption
is true, we have 
that there exists a constant $C'$ such that $\|v_n\| > C'$. If $\eta$ is 
small enough we have that $\eta C <1$, therefore $\|v_n \| > \|v_{n,2}\|$. 
Then we have 
\[
\| v_{n,1} \| = \| v_{n,1} + v_{n,2}  - v_{n,2} \| \geq \| v_n \|- \|v_{n,2}\| \geq C' - \eta C \|v_{n,1}\|. 
\]
Which yields
\[\| v_{n,1} \| \geq \frac{C'}{1+ \eta C}, \]
therefore we have that  $\|v_{n,1}\| >0$.

Using the reverse triangular inequality we have 
\[
| \|v_{n,1} \| - \|v_{n,1}  + v_{n,2}\| | \leq \| v_{n,2}\| \leq \eta C \| v_{n,1}\|.
\]

On the other hand, if $\eta$ is small enough we have that $\eta C <1$, therefore 
$\|v_{n,2}\| <\|v_{n,1}\| $ and consequently we have that 
\[
\|v_{n,1} +  v_{n,2} \| \geq | \|v_{n,1}\| - \|  v_{n,2} \|| \geq (1-C \eta)\|v_n\|.
\]

Applying the two last inequalities to equation (\ref{equation proof of 
prop close to universal}) the result follows. 
\end{proof} 

Using lemma \ref{prop ratios acotats implica distancies acotades} we have 
that 
\begin{equation}
\label{eq:vn = vn,1 + O(eta)} 
\frac{v_n}{\|v_n\|} = \frac{v_{n,1}}{\|v_{n,1}\|} + O(\eta). 
\end{equation}

Using theorem \ref{reduction alphas to dynamics of renormalization} 
and the definition of the equivalence relation $\sim$ 
follows 
\begin{equation}
\label{eq: thm conseq3 0}
\frac{\alpha_n'(c_\eta,\omega_0)}{ \alpha_{n-1}'(c_\eta,\omega_0)} = 
\deltabf^{-1} \cdot \frac{
\displaystyle m \left( DG_1\left(\omega_{n-1}, f^*_{1} , 
\frac{v_{n-1}}{\|v_{n-1}\|} \right) \right)
}{ \rule{0pt}{4ex} \displaystyle 
m \left( DG_1\left(\omega_{n-2},  f^*_1,  \frac{v_{n-2}}
{\|v_{n-2}\|} \right) \right)
}\cdot 
\left\| D\TT_{\omega_{n-2}}(f_2^*) \frac{ v_{n-2}}{\|v_{n-2}\|}.
\right\| + O(\rho^n)
\end{equation} 
with $DG_1$, $m$  and $f^*_1$ given by the hypotheses of the theorem. 

Since the hypotheses of theorem \ref{reduction alphas to 
dynamics of renormalization} are satisfied, we have that 
$m \left( DG_1\left(\omega_{k}, f^*_{1} , \cdot \right)\right)$  and 
$\|D\TT_{\omega_k}(f_2^*)(\cdot)\|$ are 
differentiable functions. 
Using this and  (\ref{eq:vn = vn,1 + O(eta)})  we obtain 
\begin{equation} 
\label{eq: thm conseq3 1}
m \left( DG_1\left(\omega_{k}, f^*_{1} , 
\frac{v_{k}}{\|v_{k}\|} \right) \right)  = m \left( DG_1\left(\omega_{k}, f^*_{1} , 
\frac{v_{k,1}}{\|v_{k,1}\|} \right)\right) + O(\eta), \text{ for } k=n-2,n-1, 
\end{equation}
and
\begin{equation} 
\label{eq: thm conseq3 2}
\left\| D\TT_{\omega_{n-2}}(f_2^*) \frac{ v_{n-2}}{\|v_{n-2}\|}.
\right\| = \left\| D\TT_{\omega_{n-2}}(f_2^*) \frac{ v_{n-2,1}}{\|v_{n-2,1}\|}.
\right\| + O(\eta). 
\end{equation}
Replacing (\ref{eq: thm conseq3 1}) and (\ref{eq: thm conseq3 2}) into 
(\ref{eq: thm conseq3 0}) follows easily
\[\frac{\alpha_n'(c_\eta,\omega_0)}{ \alpha_{n-1}'(c_\eta,\omega_0)} = 
\deltabf^{-1} \cdot \frac{
\displaystyle m \left( DG_1\left(\omega_{n-1}, f^*_{1} , 
\frac{v_{n-1,1}}{\|v_{n-1,1}\|} \right) \right)
}{ \rule{0pt}{4ex} \displaystyle 
m \left( DG_1\left(\omega_{n-2},  f^*_1,  \frac{v_{n-2,1}}
{\|v_{n-2,1}\|} \right) \right)
}\cdot 
\left\| D\TT_{\omega_{n-2}}(f_2^*) \frac{ v_{n-2,1}}{\|v_{n-2,1}\|}.
\right\| + O(\eta) + O(\rho^n). \]

Using this and theorem \ref{reduction alphas to dynamics of renormalization} 
on the family $c_0$ we have 
\[\frac{\alpha_n'(c_0,\omega_0)}{ \alpha_{n-1}'(c_0,\omega_0)} = 
\deltabf^{-1} \cdot \frac{
\displaystyle m \left( DG_1\left(\omega_{n-1}, f^*_{1} , 
\frac{v_{n-1,1}}{\|v_{n-1,1}\|} \right) \right)
}{ \rule{0pt}{4ex} \displaystyle 
m \left( DG_1\left(\omega_{n-2},  f^*_1,  \frac{v_{n-2,1}}
{\|v_{n-2,1}\|} \right) \right)
}\cdot 
\left\| D\TT_{\omega_{n-2}}(f_2^*) \frac{ v_{n-2,1}}{\|v_{n-2,1}\|}.
\right\| + O(\rho^n). \]

Using the two last equations the result follows. 
\end{proof}


\bibliographystyle{plain}

\end{document}

%% file: EsquemBifPar.pdf_t
\begin{picture}(0,0)%
\includegraphics{EsquemBifPar.pdf}%
\end{picture}%
\setlength{\unitlength}{4144sp}%
\begingroup\makeatletter\ifx\SetFigFont\undefined%
\gdef\SetFigFont#1#2#3#4#5{%
  \reset@font\fontsize{#1}{#2pt}%
  \fontfamily{#3}\fontseries{#4}\fontshape{#5}%
  \selectfont}%
\fi\endgroup%
\begin{picture}(7992,4008)(346,-4405)
\put(2791,-3796){\makebox(0,0)[lb]{\smash{{\SetFigFont{12}{14.4}{\familydefault}{\mddefault}{\updefault}{\color[rgb]{0,0,0}$f_2$}%
}}}}
\put(2296,-3796){\makebox(0,0)[lb]{\smash{{\SetFigFont{12}{14.4}{\familydefault}{\mddefault}{\updefault}{\color[rgb]{0,0,0}$s_1$}%
}}}}
\put(1621,-3796){\makebox(0,0)[lb]{\smash{{\SetFigFont{12}{14.4}{\familydefault}{\mddefault}{\updefault}{\color[rgb]{0,0,0}$f_1$}%
}}}}
\put(361,-3751){\makebox(0,0)[lb]{\smash{{\SetFigFont{12}{14.4}{\familydefault}{\mddefault}{\updefault}{\color[rgb]{0,0,0}$\alpha$}%
}}}}
\put(901,-3796){\makebox(0,0)[lb]{\smash{{\SetFigFont{12}{14.4}{\familydefault}{\mddefault}{\updefault}{\color[rgb]{0,0,0}$s_0$}%
}}}}
\put(3196,-3796){\makebox(0,0)[lb]{\smash{{\SetFigFont{12}{14.4}{\familydefault}{\mddefault}{\updefault}{\color[rgb]{0,0,0}$s_2$}%
}}}}
\put(451,-1816){\makebox(0,0)[lb]{\smash{{\SetFigFont{12}{14.4}{\familydefault}{\mddefault}{\updefault}{\color[rgb]{0,0,0}$\eps$}%
}}}}
\put(4231,-781){\makebox(0,0)[lb]{\smash{{\SetFigFont{12}{14.4}{\familydefault}{\mddefault}{\updefault}{\color[rgb]{0,0,0}$\text{Reducibility loss bifurcations}$}%
}}}}
\put(4231,-1006){\makebox(0,0)[lb]{\smash{{\SetFigFont{12}{14.4}{\familydefault}{\mddefault}{\updefault}{\color[rgb]{0,0,0}$\text{Boundary to chaos}$}%
}}}}
\put(4546,-3751){\makebox(0,0)[lb]{\smash{{\SetFigFont{12}{14.4}{\familydefault}{\mddefault}{\updefault}{\color[rgb]{0,0,0}$\alpha$}%
}}}}
\put(7336,-3886){\makebox(0,0)[lb]{\smash{{\SetFigFont{12}{14.4}{\familydefault}{\mddefault}{\updefault}{\color[rgb]{0,0,0}$s_2$}%
}}}}
\put(6436,-3886){\makebox(0,0)[lb]{\smash{{\SetFigFont{12}{14.4}{\familydefault}{\mddefault}{\updefault}{\color[rgb]{0,0,0}$s_1$}%
}}}}
\put(5761,-3886){\makebox(0,0)[lb]{\smash{{\SetFigFont{12}{14.4}{\familydefault}{\mddefault}{\updefault}{\color[rgb]{0,0,0}$f_1$}%
}}}}
\put(5131,-3886){\makebox(0,0)[lb]{\smash{{\SetFigFont{12}{14.4}{\familydefault}{\mddefault}{\updefault}{\color[rgb]{0,0,0}$s_0$}%
}}}}
\put(4636,-1906){\makebox(0,0)[lb]{\smash{{\SetFigFont{12}{14.4}{\familydefault}{\mddefault}{\updefault}{\color[rgb]{0,0,0}$\eps$}%
}}}}
\put(6886,-3886){\makebox(0,0)[lb]{\smash{{\SetFigFont{12}{14.4}{\familydefault}{\mddefault}{\updefault}{\color[rgb]{0,0,0}$f_2$}%
}}}}
\put(496,-1546){\makebox(0,0)[lb]{\smash{{\SetFigFont{12}{14.4}{\familydefault}{\mddefault}{\updefault}{\color[rgb]{0,0,0}$\text{Bifurcation diagram for }\omega$}%
}}}}
\put(4726,-1546){\makebox(0,0)[lb]{\smash{{\SetFigFont{12}{14.4}{\familydefault}{\mddefault}{\updefault}{\color[rgb]{0,0,0}$\text{Bifurcation diagram for }2\omega$}%
}}}}
\put(3601,-4336){\makebox(0,0)[lb]{\smash{{\SetFigFont{12}{14.4}{\familydefault}{\mddefault}{\updefault}{\color[rgb]{0,0,0}$\text{Affine relationship?}$}%
}}}}
\put(4231,-556){\makebox(0,0)[lb]{\smash{{\SetFigFont{12}{14.4}{\familydefault}{\mddefault}{\updefault}{\color[rgb]{0,0,0}$\text{Period doubling bifurcations}$}%
}}}}
\end{picture}%

%% file: EsquemaVar3.pdf_t
\begin{picture}(0,0)%
\includegraphics{EsquemaVar3.pdf}%
\end{picture}%
\setlength{\unitlength}{4144sp}%
\begingroup\makeatletter\ifx\SetFigFontNFSS\undefined%
\gdef\SetFigFontNFSS#1#2#3#4#5{%
  \reset@font\fontsize{#1}{#2pt}%
  \fontfamily{#3}\fontseries{#4}\fontshape{#5}%
  \selectfont}%
\fi\endgroup%
\begin{picture}(4125,3579)(1066,-4483)
\put(2386,-3391){\makebox(0,0)[lb]{\smash{{\SetFigFontNFSS{12}{14.4}{\familydefault}{\mddefault}{\updefault}{\color[rgb]{0,0,0}$f_{k}^{(n+1)}$}%
}}}}
\put(3376,-1141){\makebox(0,0)[lb]{\smash{{\SetFigFontNFSS{12}{14.4}{\familydefault}{\mddefault}{\updefault}{\color[rgb]{0,0,0}$\RR(c(\alpha,0))$}%
}}}}
\put(4501,-1231){\makebox(0,0)[lb]{\smash{{\SetFigFontNFSS{12}{14.4}{\familydefault}{\mddefault}{\updefault}{\color[rgb]{0,0,0}$c(\alpha,0)$}%
}}}}
\put(4951,-1591){\makebox(0,0)[lb]{\smash{{\SetFigFontNFSS{12}{14.4}{\familydefault}{\mddefault}{\updefault}{\color[rgb]{0,0,0}$\Sigma_{1}$}%
}}}}
\put(4951,-2176){\makebox(0,0)[lb]{\smash{{\SetFigFontNFSS{12}{14.4}{\familydefault}{\mddefault}{\updefault}{\color[rgb]{0,0,0}$\Sigma_{2}$}%
}}}}
\put(2296,-2041){\makebox(0,0)[lb]{\smash{{\SetFigFontNFSS{12}{14.4}{\familydefault}{\mddefault}{\updefault}{\color[rgb]{0,0,0}$f_{n-1}^{(n)}$}%
}}}}
\put(4951,-4291){\makebox(0,0)[lb]{\smash{{\SetFigFontNFSS{12}{14.4}{\familydefault}{\mddefault}{\updefault}{\color[rgb]{0,0,0}$W^s(\RR)$}%
}}}}
\put(2296,-1141){\makebox(0,0)[lb]{\smash{{\SetFigFontNFSS{12}{14.4}{\familydefault}{\mddefault}{\updefault}{\color[rgb]{0,0,0}$\RR^n(c(\alpha,0))$}%
}}}}
\put(4006,-3571){\makebox(0,0)[lb]{\smash{{\SetFigFontNFSS{12}{14.4}{\familydefault}{\mddefault}{\updefault}{\color[rgb]{0,0,0}$f_0^{(n)} $}%
}}}}
\put(1441,-4381){\makebox(0,0)[lb]{\smash{{\SetFigFontNFSS{12}{14.4}{\familydefault}{\mddefault}{\updefault}{\color[rgb]{0,0,0}$\Phi$}%
}}}}
\put(5176,-3706){\makebox(0,0)[lb]{\smash{{\SetFigFontNFSS{12}{14.4}{\familydefault}{\mddefault}{\updefault}{\color[rgb]{0,0,0}$\Sigma_{n}$}%
}}}}
\put(1081,-1051){\makebox(0,0)[lb]{\smash{{\SetFigFontNFSS{12}{14.4}{\familydefault}{\mddefault}{\updefault}{\color[rgb]{0,0,0}$W^u(\RR)$}%
}}}}
\end{picture}%